\documentclass[]{interact}

\usepackage{epstopdf}
\usepackage[caption=false]{subfig}
\usepackage{mathtools}
\usepackage{xcolor}
\ifdefined\showedits\newcommand{\rev}[1]{{\color{red}#1}}\else\newcommand{\rev}[1]{#1}\fi
\usepackage{float}
\usepackage{enumitem}
\usepackage{array}
\usepackage{tabularx}
\usepackage{algorithm}
\usepackage{algpseudocode}
\usepackage[numbers,sort&compress]{natbib}
\bibpunct[, ]{[}{]}{,}{n}{,}{,}

\usepackage[hidelinks]{hyperref}
\theoremstyle{plain}
\newtheorem{theorem}{Theorem}[section]
\newtheorem{lemma}[theorem]{Lemma}
\newtheorem{corollary}[theorem]{Corollary}
\newtheorem{proposition}[theorem]{Proposition}
\theoremstyle{definition}
\newtheorem{definition}[theorem]{Definition}
\newtheorem{example}[theorem]{Example}
\newtheorem{remark}[theorem]{Remark}

\newcommand{\R}{\mathbb R}
\def\epi{\textnormal{epi}\,}
\def\gph{\textnormal{gph}\,}
\def\cone{\textnormal{cone}\,}
\def\proj{\textnormal{proj}\,}
\def\dom{\mathop{\rm dom\,}}
\def\co{\mathop{\rm co\,}}
\def\clco{\mathop{\rm clco\,}}
\def\cl{\mathop{\rm cl\,}}
\def\argmin{\mathop{\rm argmin\,}}

\begin{document}

\articletype{ARTICLE}

\title{Convergence and Stability of a Catching-Up Algorithm for Differential Inclusions with Maximally Monotone Operators}

\author{
\name{Tan H. Cao\textsuperscript{a} and Hassan Saoud\textsuperscript{b}\thanks{CONTACT Tan H. Cao. Email: tan.cao@stonybrook.edu}
}
\affil{\textsuperscript{a}Department of Applied Mathematics and Statistics, State University of New York--Korea, Incheon, Republic of Korea;\\ 
\textsuperscript{b}Department of Mathematics and Natural Sciences, Gulf University for Science and Technology, Hawally 32093, Kuwait}
}

\maketitle

\begin{abstract}
We study a catching-up algorithm with approximate projections for
finite-dimensional differential inclusions governed by a maximally monotone
operator and a continuous perturbation. Through the canonical decomposition
of the operator into a regular part and the normal cone to the closure of
its domain, the problem becomes a sweeping process, discretized by an explicit
predictor step and a feasible approximate projection. For variable time steps, the approximations converge, up to a
subsequence, to a solution on every finite horizon when the mesh size and
the accumulated projection-error budget vanish and the nodes stay in a
common compact convex set. This set comes from linear growth on
finite horizons and, without any growth condition, by dissipativity of the
projected field at a feasible center, which confines the nodes to a fixed
ball uniformly in time for small steps and small projection errors. Under
either condition, locally Lipschitz perturbations give global
well-posedness, convergence of the whole family, and\rev{, for a common
initial state,} a trajectory-error
estimate of order $O(\sqrt{h+E})$, where $h$ is the maximal step and $E$ the
accumulated budget; sparse projection errors are allowed. An elastic
network with unilateral stops illustrates the results, and numerical
experiments compare the algorithm with the full-resolvent discretization.
\end{abstract}


\begin{amscode}
Primary 34A60; Secondary 47H05, 65L20, 65L70, 49J53
\end{amscode}

\section{Introduction}\label{sec:intro}

\subsection{Some background}

Evolution inclusions governed by maximally monotone operators provide a
classical framework for dissipative dynamics with nonsmooth constitutive
laws and unilateral constraints. Their theory grew together with nonlinear
semigroups and implicit time discretization. The generation theorem
of~\cite{CrandallLiggett1971} establishes a resolvent
approximation of nonlinear evolutions in Banach spaces, and the
Hilbert-space theory of maximally monotone operators and their
perturbations is developed in~\cite{Brezis1973,benilan-brezis72}; see
also~\cite{barbu2010}. Through subdifferentials and normal cones, this
theory is connected with convex optimization and with constrained
mechanics. Differential inclusions, viability, and their mechanical
applications are treated systematically
in~\cite{AubinCellina1984,Marques1993}.

A closely related development is Moreau's sweeping process, in which a
moving convex set determines both the admissible states and the normal
reaction. \rev{The catching-up construction
of~\cite{Moreau1971,Moreau1972,Moreau1977} approximates this evolution by
successive projections}, and its
numerical aspects are discussed in~\cite{Moreau1999}. Subsequent work treats
set-valued perturbations~\cite{Castaing1993} and nonconvex sweeping
processes with perturbations and solutions of bounded
variation~\cite{Thibault2006}, and regularization methods and their
relation to the catching-up construction are surveyed in~\cite{Nacry2019}.
\rev{When the constraint set is described by smooth inequalities, a
prediction--correction method based on projections onto convex
linearizations is developed in~\cite{VenelScheme2011}, and a square-root
convergence rate is established in~\cite{BernicotVenel2013} under suitable
regularity assumptions on the geometry and the perturbation. More recently,
catching-up algorithms in which the projection at each step is computed
only up to a tolerance have been studied
in~\cite{Vilches2024,GarridoLioiVilches2025}.}

The goal of this paper is to study such a catching-up scheme with
approximate projections for differential inclusions governed by a maximally
monotone operator \(A\) and a continuous perturbation \(f\), in finite
dimensions and on a fixed convex constraint set that may be unbounded.
\rev{We do not require a globally Lipschitz bound on the minimal norm of the
perturbation, and we also treat fields with superlinear growth.} In this case, the bounded region required
by the analysis comes from a dissipativity condition on the projected
field. \rev{Under thresholds on the step size and on the pointwise
tolerance, this condition confines the iterates to a fixed ball, uniformly
in time.}
Moreover, the projection errors are measured by an accumulated budget on
nonuniform meshes, which allows sparse inexact projections. In this
setting, we prove the convergence of the scheme on finite horizons\rev{, up
to a subsequence. For locally Lipschitz perturbations, we also prove global
well-posedness and\rev{, for a common initial state,} a trajectory-error estimate of order
\(\sqrt{h+E}\)}, where \(h\) is the mesh size and \(E\) the accumulated
budget.

\subsection{The problem and the scheme}

Throughout the paper, the ambient space is \(\mathbb R^n\) and the
constraint set is fixed. Let \(A:\mathbb R^n\rightrightarrows\mathbb R^n\) be
maximally monotone, set \(C:=\cl(\dom A)\), and let \(f:C\to\mathbb R^n\) be
continuous. We consider
\begin{equation*}\label{eq:main-DI}
\begin{cases}
\dot x(t)\in f(x(t))-A(x(t)) & \text{for a.e. }t\ge0,\\
x(0)=x_0\in C.
\end{cases}
\tag{$P_A$}
\end{equation*}
We assume that \(\operatorname{int}(\dom A)\ne\emptyset\), and we use the
classical interior-limit representation of \(A\); see
\cite[Theorem~12.67]{Rock} and the related structural results
in~\cite{borwein2013structure}. Writing \(A_0(x)\) for the finite limits of
the single-valued interior values of \(A\) along sequences approaching
\(x\), this representation gives \(A=G+N_C\), with \(G:=\clco A_0\). The
additional assumption~\eqref{ass:G} requires \(G\) to be upper
semicontinuous relative to \(C\), with nonempty compact convex values; the
representation theorem alone does not provide this boundary regularity, as
Example~\ref{ex:log-boundary} shows. With \(F:=f-G\),
problem~\eqref{eq:main-DI} becomes
\begin{equation}\label{eq:main-decomp}
\dot x(t)\in F(x(t))-N_C(x(t)),\qquad x(t)\in C.
\end{equation}

This representation suggests the predictor--projection update
\[
y_{k+1}=x_k+h_k w_k,\qquad w_k\in F(x_k),
\qquad x_{k+1}\in\proj_C^{\varepsilon_k}(y_{k+1}),
\]
where the corrected point belongs to \(C\) and satisfies
\(\|x_{k+1}-y_{k+1}\|^2\le d_C(y_{k+1})^2+\varepsilon_k\). The selected value
of \(G\) is evaluated exactly, and the tolerance concerns only the
projection. For exact projections, the update has a forward--backward form,
since \((\mathrm{Id}+h N_C)^{-1}=\proj_C\); it is therefore of computational
interest when a selection of \(G\) and the projection onto \(C\) are
accessible. In contrast, the full-resolvent update evaluates
\((\mathrm{Id}+h A)^{-1}\) and treats the whole monotone operator
implicitly. Resolvent methods also underlie the proximal-point algorithm
for finding zeros of maximally monotone operators~\cite{Rockafellar1976},
whose continuous and discrete asymptotic theory is reviewed
in~\cite{PeypouquetSorin2010}; for splitting and proximal computations, we
refer the reader to~\cite{Bauschke,ParikhBoyd2014}. Our analysis concerns,
instead, the approximation of trajectories on finite horizons as the mesh
is refined.

\rev{On the common ground of a fixed convex set in finite dimension, our
assumptions differ from those of~\cite{Vilches2024,GarridoLioiVilches2025},
where approximate projections are studied for moving constraints in Hilbert
spaces, including nonconvex regularity classes. In~\cite{Vilches2024}, the
corrections are feasible, while~\cite{GarridoLioiVilches2025} also permits
infeasible corrections with a separate distance tolerance. Both works use a
globally Lipschitz majorant of the minimal norm of the perturbation and a
controlled measurable selection, and their algorithms use uniform steps \(h\)
and the same tolerance at all nodes of a mesh, with \(\varepsilon/h^2\to0\);
the second one also requires a feasibility tolerance \(o(h)\). Under all-pairs
one-sided Lipschitz conditions, their prox-regular results give quantitative
error estimates. Their geometric setting is broader than ours, since here
the constraint set is fixed and convex and the space is finite-dimensional.
Within this setting, we combine the one-sided Lipschitz stability argument
with the canonical splitting \(A=G+N_C\), dissipative confinement, and the
control of accumulated projection errors on nonuniform meshes.}

\subsection{Main results}

The first difficulty is to find a bounded region on which the drift can be
controlled, since \(C\) may be unbounded. Our convergence theorem requires
a compact convex set containing all the nodes, and the rate theorem also
requires the comparison solutions to stay in this set. We provide such a
set with and without the linear growth condition~\eqref{A2}. Under~\eqref{A2},
it exists on each finite horizon as soon as the accumulated budgets and the
initial states are bounded, with no threshold on the step size.
Without~\eqref{A2}, we impose the dissipativity condition~\eqref{A3} on the
projected effective field, so that it accounts for the interaction between
\(f\), \(G\), and the constraint reaction, either at a center \(p\in C\) or
at \(p=0\) under a compatibility condition.
Theorems~\ref{thm:confinement-anchor} and~\ref{thm:confinement-HR} then keep
the nodes in a fixed ball, uniformly in time. \rev{Our confinement results
without~\eqref{A2} thus impose a threshold on the step size and a
threshold on the pointwise tolerance.}

The second point concerns the distribution of the projection tolerances
across the mesh. We measure them by two quantities, defined in
Definition~\ref{def:scheme_energy}: the accumulated budget \(E_\nu(T)\),
which is the sum of the ratios \(\varepsilon_k/h_k\), and the pointwise
budget \(Q_\nu(T)\), which is the largest ratio \(\varepsilon_k/h_k^2\). For
uniform steps and a common tolerance, \(E_\nu(T)=TQ_\nu(T)\), so the two
budgets vanish together. However, a single tolerance \(c^2h^2\) in one cell,
with exact projections in all the other cells, gives \(E_\nu(T)=c^2h\to0\),
while \(Q_\nu(T)=c^2\) does not vanish. Our convergence theorem,
Theorem~\ref{thm:limit_solution}, requires only \(E_\nu(T)\to0\), together
with the refinement of the meshes, the convergence of the initial states,
and a common compact convex region containing the nodes; it therefore
allows such sparse errors. The key device is the exact shadow point
\(\widehat x_{k+1}=\proj_C(y_{k+1})\): the normals are evaluated there, and
the residual \(x_{k+1}-\widehat x_{k+1}\), after scaling, has a squared
\(L^2\) norm at most \(E_\nu(T)\). When confinement is provided by
Theorem~\ref{thm:confinement-anchor}, its pointwise tolerance threshold is
still required.

When the nodes and the comparison solutions lie in a common compact convex
region, the local all-pairs one-sided Lipschitz condition yields, through a
differential energy comparison, the finite-horizon estimate
\[
\|x_\mu-x\|_{L^\infty(0,T)}
\le C_T\left(\|x_{\nu,0}-x_0\|
             +\sqrt{h_\nu+E_\nu(T)}\right),
\]
with a constant independent of the mesh, whose coefficients are given
explicitly in Theorem~\ref{thm:trajectory-rate}. This estimate keeps the
classical square-root order in the mesh size, its constant may depend on
\(T\), and it requires no bound on the pointwise budget \(Q_\nu(T)\). In the
canonical setting, the local Lipschitz continuity of \(f\) and the
monotonicity of \(G\) provide the all-pairs condition. This condition also
gives uniqueness, stability, and the convergence of the whole family, and
together with the bounds above, it yields global well-posedness in
Theorem~\ref{thm:global-wellposed}.

Quantitative estimates of this kind have an earlier literature. For
unconstrained finite-dimensional inclusions, the authors
of~\cite[Theorem~4.3 and Corollary~4.4]{DonchevFarkhi1998} give Euler
estimates under one-sided Lipschitz and additional continuity assumptions.
Note that their condition matches each value with a suitable value at the
other state, whereas our all-pairs condition controls every pair. For
maximally monotone evolutions in Hilbert spaces, the authors
of~\cite[Propositions~4.4--4.5]{BastienSchatzman2002} obtain an
\(O(\sqrt h)\) bound for exact resolvents and an \(O(h)\) bound for normal
cones, under their Lipschitz and temporal regularity hypotheses. Applying
the latter to our splitting would require these hypotheses for the
effective field \(F=f-G\). However, the local Lipschitz continuity of \(f\)
alone does not make \(F\) single-valued and Lipschitz.

We illustrate the framework with an overdamped elastic network with
unilateral lower stops. This application has a noncompact constraint set, a
nontrivial canonical splitting, and dissipativity at a feasible center with
a positive constant, while it violates global linear growth. Two numerical
experiments examine complementary aspects of the scheme. The first one tests
the sensitivity of the scheme to distributed and sparse projection errors
with equal accumulated budgets, and the second one compares the
computational work of the split and full-resolvent discretizations at
matched trajectory-error targets. The results illustrate the theoretical
estimates; they also show an advantage of splitting in this implementation,
but not a better convergence order.

The layout of the paper is as follows. Section~\ref{sec:prelim} collects the
preliminaries on maximally monotone operators and their canonical
decomposition. In Section~\ref{sec:setting}, we introduce the framework and
the assumptions, and we describe three ways to obtain a compact region
containing the nodes. Section~\ref{sec:energy} is devoted to the continuous
energy estimate, the catching-up scheme, and the discrete confinement
without linear growth. In Section~\ref{sec:convergence}, we prove the
convergence on finite horizons and global well-posedness, and
Section~\ref{sec:quantitative} establishes the stability and
trajectory-error estimates. Section~\ref{sec:noncompact-springs} develops
the elastic network and its numerical experiments. Finally,
Section~\ref{sec:conc} summarizes the results and discusses future
research. The implementation details and the numerical data are collected
in the supplementary material.

\section{Notation and preliminaries}
\label{sec:prelim}

We begin this section by recalling the notation and the facts from convex
analysis and monotone operator theory used throughout the paper. The
central object is the canonical decomposition~\eqref{eq:decomp} of a
maximally monotone operator into a regular part and the normal cone to the
closure of its domain. The regular part is well behaved in the interior of
the domain, but it may degenerate on the boundary.

Throughout the paper, $\|\cdot\|$ and $\langle\cdot,\cdot\rangle$ denote the
Euclidean norm and inner product on $\R^n$, while $B$ and $\overline B$
denote the open and closed unit balls. We write $B(x;\rho)$ and
$\overline B(x;\rho)$ for the corresponding balls of center $x$ and radius
$\rho>0$. For $C\subset\R^n$, the symbols $\operatorname{int}C$,
$\partial C$, $\cl C$, $\co C$, and $\clco C:=\cl(\co C)$ denote,
respectively, its interior, boundary, closure, convex hull, and closed
convex hull.

Let $C\subset\R^n$ be nonempty, closed, and convex. For $x\in\R^n$, set
\[
d_C(x):=\min_{z\in C}\|x-z\|,
\qquad
\proj_C(x):=\argmin_{z\in C}\|x-z\|.
\]
By the projection theorem, $\proj_C(x)$ is a singleton, which we identify
with its unique element; see \cite[Theorem~3.16]{Bauschke}.
For $\varepsilon\ge0$, the approximate projection is
\[
\proj_C^\varepsilon(y)
 :=\{z\in C:\|z-y\|^2\le d_C^2(y)+\varepsilon\}.
\]
For every $y\in\R^n$, this set is nonempty, and it is compact because the
space is finite-dimensional. Moreover, $\proj_C^0(y)=\{\proj_C(y)\}$.

For $T>0$, the spaces $AC([0,T];\R^n)$ and $C([0,T];\R^n)$ consist,
respectively, of the absolutely continuous and the continuous functions on
$[0,T]$. We write $AC_{\mathrm{loc}}([0,\infty);\R^n)$ for the functions
whose restriction to every compact interval $[0,T]$ belongs to
$AC([0,T];\R^n)$. The Lebesgue spaces of $\R^n$-valued functions are
denoted by $L^p(0,T;\R^n)$, $1\le p\le\infty$, and by $L^p(0,T)$ in the
scalar case. We write $\rightharpoonup^\ast$ for weak-$^\ast$ convergence
in $L^\infty(0,T;\R^n)$.

For an extended-real-valued function
$\varphi:\R^n\to\R\cup\{+\infty\}$, the effective domain and the epigraph
are
\[
\dom\varphi:=\{x\in\R^n:\varphi(x)<+\infty\},
\qquad
\epi\varphi:=\{(x,\alpha)\in\R^n\times\R:\varphi(x)\le\alpha\}.
\]
The function $\varphi$ is proper if $\dom\varphi\neq\emptyset$ and
$\varphi(x)>-\infty$ everywhere, lower semicontinuous if $\epi\varphi$ is
closed, and convex if $\epi\varphi$ is convex. For a convex function
$\varphi$ and $x\in\dom\varphi$, the Fenchel subdifferential of $\varphi$
at $x$ is
\[
\partial\varphi(x)
 :=\{\zeta\in\R^n:
 \varphi(y)\ge\varphi(x)+\langle\zeta,y-x\rangle
 \text{ for every }y\in\R^n\}.
\]

For a nonempty closed convex set $C$, let $I_C$ be its indicator function,
equal to $0$ on $C$ and to $+\infty$ outside $C$. Its Fenchel
subdifferential is the normal-cone mapping
\[
N_C(x):=\partial I_C(x)
 =
 \begin{cases}
 \{v\in\R^n:\langle v,z-x\rangle\le0
          \text{ for every }z\in C\},&x\in C,\\
 \emptyset,&x\notin C.
 \end{cases}
\]
Consequently,
\[
p=\proj_C(y)
\quad\Longleftrightarrow\quad
y-p\in N_C(p)
\quad\Longleftrightarrow\quad
\langle y-p,z-p\rangle\le0
\quad\text{for every }z\in C.
\]

For $x\in C$, the tangent cone to $C$ at $x$ is \(T_C(x):=\cl\cone(C-x)\).
For a cone $K\subset\R^n$, the negative polar is
$K^\circ:=\{v\in\R^n:\langle v,w\rangle\le0
\text{ for every }w\in K\}$. Hence
\[
N_C(x)=T_C(x)^\circ,\qquad
u=\proj_{T_C(x)}u+\proj_{N_C(x)}u,\qquad
\big\langle\proj_{T_C(x)}u,\proj_{N_C(x)}u\big\rangle=0.
\]
The last two identities express Moreau's decomposition for mutually polar
closed convex cones; see \cite[Theorem~6.30]{Bauschke}.

A set-valued mapping $A:\R^n\rightrightarrows\R^n$ has graph and domain \(\gph A:=\{(x,y)\in\R^n\times\R^n:y\in A(x)\}\) and \(\dom A:=\{x\in\R^n:A(x)\neq\emptyset\}\).
It is monotone if
\[
\langle y_1-y_2,x_1-x_2\rangle\ge0
\quad
\text{for all }(x_i,y_i)\in\gph A,\ i=1,2,
\]
and maximally monotone if its graph has no proper monotone extension. Set
$C:=\cl(\dom A)$. For a maximally monotone operator, $C$ is convex and
$\dom A$ is nearly convex; see \cite[Theorems~12.37 and~12.41]{Rock}.
Moreover, the values of $A$ are closed and convex, possibly unbounded or
empty.

A fundamental example is $A=\partial\varphi$, where $\varphi$ is proper,
lower semicontinuous, and convex. In this case, \(\dom(\partial\varphi)\subseteq\dom\varphi\subseteq\cl(\dom\varphi)=\cl(\dom\partial\varphi)\).
For $x\in\dom A$, the operator $A$ is locally bounded at $x$ if and only if
$x\in\operatorname{int}(\dom A)$; see \cite[Corollary~12.38]{Rock}.

From now on, we assume that $\operatorname{int}(\dom A)\neq\emptyset$, and
we let \(E:=\{x\in\operatorname{int}(\dom A):A(x)\text{ is a singleton}\}\).
By \cite[Corollary~12.66]{Rock}, the points of differentiability of $A$
belong to $E$ and are dense in $\dom A$. Identifying $A(x)$ with its unique
element for $x\in E$, we define
\begin{equation}\label{eq:A0}
A_0(x):=\mathop{\rm Limsup}_{y\to x,\ y\in E}A(y).
\end{equation}
This is the Painlev\'e--Kuratowski outer limit of $A$ relative to $E$; see
\cite[Chapter~5]{Rock}. Thus $v\in A_0(x)$ if and only if there exist
$x_k\in E$ with $x_k\to x$ and $A(x_k)\to v$. In particular, $A_0(x)$
contains only finite limits. Since $E$ lies in the interior of $\dom A$,
the definition of $A_0$ does not use the values of $A$ on the boundary of
$C$. For instance, if $A=\partial|\cdot|$ on $\R$, then
$A_0(0)=\{-1,1\}$, whereas $A(0)=[-1,1]$. Here the closed convex hull of
$A_0(0)$ recovers $A(0)$. In general, the normal cone to $C$ must also be
added, as the following decomposition shows. By
\cite[Theorem~12.67]{Rock},
\begin{equation}\label{eq:decomp}
A(x)=\clco A_0(x)+N_C(x),\qquad x\in\R^n.
\end{equation}
The same theorem gives \(\dom A_0=\dom A\) and \(\cl E=\cl(\dom A_0)=\cl(\dom A)=C\),
and $A_0=A$ on $E$. For later use, we set \(G:=\clco A_0\).

We first show that, on the interior of the domain, the regular part $G$
inherits the regularity of the maximally monotone operator $A$.
{The boundary version of the following statement is \rev{given} in
\cite[Proposition~4.2]{SaoudTheraDao2026}\rev{, within the uniform
boundedness framework of that paper}. 
 \rev{Without uniform boundedness}, it may fail, as
Example~\ref{ex:log-boundary} shows; this is why the boundary regularity
is imposed separately in~\eqref{ass:G}.}

\begin{proposition}[Interior regularity of $G$]
\label{prop:usc-locbnd}
The mapping $G$ has nonempty compact convex values and is locally bounded
and upper semicontinuous on $\operatorname{int}(\dom A)$.
\end{proposition}

\begin{proof}
We split the proof into two steps.

\textbf{Step~1}. \emph{Identification on the interior.}
By the near convexity of $\dom A$, we have
$\operatorname{int}(\dom A)=\operatorname{int}C$. Hence
$N_C(x)=\{0\}$ for every $x\in\operatorname{int}(\dom A)$, and
the decomposition~\eqref{eq:decomp} yields \(G(x)=A(x)\) for \(x\in\operatorname{int}(\dom A)\).

\textbf{Step~2}. \emph{Regularity properties.}
A maximally monotone operator is outer semicontinuous and closed-convex-valued,
and it is locally bounded on the interior of its domain; see
\cite[Exercise~12.8, Corollary~12.38]{Rock}. Under local boundedness,
outer semicontinuity is equivalent to upper semicontinuity
\cite[Theorem~5.19]{Rock}. Moreover, bounded closed values are compact
because the space is finite-dimensional.
Hence $G$ has all the properties stated in the proposition.
\end{proof}

For subdifferentials of convex functions, let $A=\partial\varphi$ and
assume that $U:=\operatorname{int}(\dom\varphi)\neq\emptyset$. Then
$\varphi|_U$ is finite, convex, and locally Lipschitz, and
$G(x)=\partial\varphi(x)$ for $x\in U$. This set is also the Clarke
subdifferential of $\varphi|_U$; see \cite[Theorem~25.6]{Rock2}
and~\cite{clarke1}. At a boundary point, however, $G$ may even be empty, as
the next example shows.

\begin{example}[Failure of boundary regularity]
\label{ex:log-boundary}
Let $\varphi(x)=-\log x$ for $x>0$ and $\varphi(x)=+\infty$ for $x\le0$.
Then $A=\partial\varphi$ is maximally monotone, with \(\dom A=(0,\infty)\), \(C=[0,\infty)\), and \(A(x)=\{-1/x\}\) for \(x>0\).
Since $-1/x_k\to-\infty$ for every positive sequence $x_k\to0$, the outer
limit in~\eqref{eq:A0} contains no finite limit at the origin; hence \(A_0(0)=G(0)=\emptyset\).
Thus $\operatorname{int}(\dom A)\neq\emptyset$ does not ensure that $G$ is
nonempty-valued or locally bounded on $C$. Boundary regularity must
therefore be imposed separately; we do so in~\eqref{ass:G} of
Section~\ref{sec:setting}.
\end{example}

\section{Problem setting and framework}\label{sec:setting}

In this section we fix the framework used in the rest of the paper. We
rewrite problem~\eqref{eq:main-DI} as a sweeping process driven by an effective
set-valued field, state our assumptions on its two parts, and explain the
role of each of them. We conclude with Lemma~\ref{lem:A3-global}, which
extends the dissipativity assumption to the whole set \(C\).

With the notation of Section~\ref{sec:prelim}, set \(F(x):=f(x)-G(x)\)
for \(x\in C\). By \eqref{eq:decomp}, problem~\eqref{eq:main-DI} becomes
\begin{equation}\label{eq:decomp-DI}
\dot x(t)\in f(x(t))-G(x(t))-N_C(x(t)),
\qquad \text{a.e. }t\ge0,
\end{equation}
or, with its initial condition,
\begin{equation*}\tag{\(P_F\)}\label{eq:F-formulation}
\left\{
\begin{aligned}
\dot x(t)&\in F(x(t))-N_C(x(t))
&&\text{a.e. }t\ge0,\\
x(0)&=x_0\in C.
\end{aligned}
\right.
\end{equation*}
A solution on \([0,T]\) is an absolutely continuous \(C\)-valued curve
satisfying these conditions.

\paragraph{Standing assumptions}
Throughout the paper, in addition to the canonical setting of
Section~\ref{sec:prelim}, we impose the following assumptions.
\begin{enumerate}[label=(A\arabic*),ref=A\arabic*,nosep]
\item \textbf{Perturbation.}\label{A1}
\(f:C\to\R^n\) is continuous (and therefore locally bounded relative
to \(C\)).
\end{enumerate}
For the canonical regular part, we assume the following.
\begin{enumerate}[label=(G),ref=G,nosep]
\item \textbf{Regular part.}\label{ass:G}
\(G(x)\) is nonempty and bounded for every \(x\in C\), and
\(G\) is upper semicontinuous relative to \(C\).
\end{enumerate}
The values of \(G\) are closed and convex by construction. Since they are
bounded by \eqref{ass:G}, they are compact. Thus \eqref{A1}
and \eqref{ass:G} make \(F\) upper semicontinuous, with nonempty compact
convex values, and bounded on each compact subset of \(C\). These bounds
are local; a growth condition at infinity appears only in \eqref{A2}
below. Moreover, since \(G(x)+N_C(x)\neq\emptyset\) for every \(x\in C\),
assumption \eqref{ass:G} and the decomposition \eqref{eq:decomp} give
\(\dom A=C\).

\paragraph{Additional assumptions}
Since \(C\) may be unbounded, the standing assumptions do not provide a
bounded region containing the discrete trajectories. We obtain such a
region from the behavior of \(F\) at infinity, through the following two
conditions, used either separately or together.
\begin{enumerate}[label=(A\arabic*),ref=A\arabic*,start=2,nosep]
\item \textbf{Linear growth.}\label{A2}
There exist \(a,b\ge0\) such that
\[
\sup_{w\in F(x)}\|w\|\le a+b\|x\|,
\qquad x\in C.
\]
\item \textbf{Tangent dissipativity at infinity.}\label{A3}
There exist a center \(p\in\R^n\) and constants \(R_\star,M,\gamma>0\)
such that
\[
\sup_{v\in\proj_{T_C(x)}F(x)}\langle x-p,v\rangle
\le M-\gamma\|x-p\|^2,
\qquad x\in C,\quad \|x-p\|\ge R_\star,
\]
where \(\proj_{T_C(x)}F(x):=
\{\proj_{T_C(x)}w:w\in F(x)\}\).
\end{enumerate}

Assumption~\eqref{A2} is the classical linear growth condition. It bounds
the drift globally and yields estimates on every finite horizon.
Assumption~\eqref{A3} is a dissipativity condition on the projected field.
When it holds with a given center \(p\), we say that \eqref{A3} holds at
\(p\). The center need not belong to \(C\).

Note that \eqref{A3} restricts \(F\) only outside the ball
\(\overline B(p;R_\star)\). Inside this ball, \(F\) is bounded by
\eqref{A1} and \eqref{ass:G}. Hence Lemma~\ref{lem:A3-global} below
provides a constant \(M_p\ge0\) such that
\begin{equation*}\tag{\(A3_p\)}\label{eq:A3-anchor}
\sup_{v\in\proj_{T_C(x)}F(x)}\langle x-p,v\rangle
\le M_p-\gamma\|x-p\|^2,
\qquad x\in C.
\end{equation*}
We refer to \eqref{eq:A3-anchor} as the global form of \eqref{A3}
at \(p\), and to \(M_p\) as a global dissipativity constant.

For \(C=\R^n\) and a single-valued field, the global form at \(p=0\)
reduces to the classical dissipativity condition
of~\cite[condition~(1.2)]{HumphriesStuart1994}, \rev{where it is used} to study
Runge--Kutta discretizations of ordinary differential equations.
Assumption~\eqref{A3} extends this condition to the projected set-valued
field. The center plays an essential role in it. Under~\eqref{A2},
assumption~\eqref{A3} at one center implies~\eqref{A3} at any other
center, possibly with a smaller rate \(\gamma\) and a larger constant
\(M\); see Remark~\ref{rem:regime-comparison}. Without~\eqref{A2}, this
implication fails, as the rotational fields of
Remark~\ref{rem:A3-anchor-comparison} show.

Let us now explain how these assumptions are used. The convergence
theorems of Sections~\ref{sec:convergence} and~\ref{sec:quantitative}
require a compact convex set \(K\subset C\) containing the nodes of the
predictor--projection scheme of Section~\ref{sec:energy}, although \(C\)
may be unbounded. Under~\eqref{A2}, such a set exists on every finite
horizon as soon as the accumulated budget is bounded; see
Corollary~\ref{cor:limit-linear-growth}. Without~\eqref{A2}, we
use~\eqref{A3}. For a center \(p\in C\),
Theorem~\ref{thm:confinement-anchor} keeps all nodes in a fixed ball
about \(p\) for small steps and small projection errors, and
Theorem~\ref{thm:confinement-HR} gives the same conclusion for \(p=0\)
under a compatibility condition. This ball does not depend on \(T\), and
superlinear growth is allowed.

\paragraph{Regularity versus linear growth}
The regularity assumption~\eqref{ass:G} and the growth
condition~\eqref{A2} are not independent. For the canonical map \(G\),
assumptions \eqref{A1}--\eqref{A2} imply~\eqref{ass:G}. Indeed, for
\(e\in E\), the bound \(\|A(e)\|\le\|f(e)\|+a+b\|e\|\) makes \(A|_E\)
locally bounded near \(C\). Hence the cluster sets \(A_0(x)\) are nonempty
and compact, and they depend upper semicontinuously on \(x\). The same
holds for their convex hulls \(G(x)=\co A_0(x)\). The converse is false,
even when~\eqref{A3} holds. For \(C=\R\), \(A=0\) and \(f(x)=-x-x^3\),
assumptions \eqref{A1} and~\eqref{ass:G} hold, but~\eqref{A2} does not.
Nevertheless, \(xf(x)=-x^2-x^4\le-x^2\). Thus~\eqref{A3} holds at \(p=0\),
and its global form~\eqref{eq:A3-anchor} holds with \(M_p=0\) and
\(\gamma=1\).

In applications, assumption \eqref{ass:G} can often be verified directly.
The next example treats two such classes: continuous monotone fields
constrained to \(C\), and subdifferentials of finite convex functions
restricted to \(C\).

\begin{example}[Classes satisfying the regularity hypothesis]
\label{ex:regular-classes}
Let \(C\subset\R^n\) be nonempty, closed and convex, with
\(\operatorname{int}C\neq\emptyset\).
If \(A=H+N_C\), where \(H:\R^n\to\R^n\) is continuous and monotone,
then \(A\) is maximally monotone \cite[Example~12.48]{Rock}.
Since \(A=H\) on \(\operatorname{int}C\), continuity and density yield
\(A_0(x)=G(x)=\{H(x)\}\) for every \(x\in C\).
More generally, let \(\psi:\R^n\to\R\) be finite and convex, and
let \(A=\partial(\psi+I_C)\). The subdifferential sum rule gives
\(A=\partial\psi+N_C\); see \cite[Corollary~10.9]{Rock}.
Since \(\partial\psi\) is locally bounded and upper semicontinuous
on \(\R^n\), the preceding interior-limit argument proves
\eqref{ass:G}. It also gives \(G(x)\subseteq\partial\psi(x)\)
on \(C\), with equality on \(\operatorname{int}C\).
If \(\psi\) is \(C^1\) on a neighborhood of \(C\), then
\(G(x)=\{\nabla\psi(x)\}\) throughout \(C\).
Equality on the boundary may fail in the nonsmooth case:
for \(\psi(x)=|x|\) and \(C=[0,1]\), we have
\(G(0)=\{1\}\neq[-1,1]=\partial\psi(0)\).
Both maps give the same operator after the addition of \(N_C\), but
\(G\) always denotes the canonical regular part in this paper.
\end{example}

\paragraph{Projected dynamics and dissipativity}
Assumption~\eqref{A3} is imposed on the projected field
\(\proj_{T_C(x)}F(x)\) rather than on \(F(x)\). The reason is that, for
absolutely continuous \(C\)-valued curves, the inclusion
in~\eqref{eq:F-formulation} is equivalent to
\(\dot x(t)\in\proj_{T_C(x(t))}F(x(t))\) almost everywhere. This
equivalence holds along curves, not pointwise. Indeed, at a time of
differentiability, \(\dot x(t)\in T_C(x(t))\), and the normal inequality
applied to \(x(t\pm h)\) gives \(\langle n,\dot x(t)\rangle=0\) for every
\(n\in N_C(x(t))\). Thus, if \(\dot x(t)=w-n\) with \(w\in F(x(t))\) and
\(n\in N_C(x(t))\), Moreau's decomposition gives
\(\dot x(t)=\proj_{T_C(x(t))}w\). Conversely, this identity implies
\(w-\dot x(t)\in N_C(x(t))\).

Let~\eqref{A3} hold at \(p\), and set \(V_p(x):=\frac12\|x-p\|^2\).
For \(v\in\proj_{T_C(x)}F(x)\), the quantity \(\langle x-p,v\rangle\) is
the derivative of \(V_p\) along \(v\). Thus~\eqref{A3} requires every
projected velocity to point inward with respect to the spheres centered
at \(p\), once \(\|x-p\|>\max\{R_\star,\sqrt{M/\gamma}\}\). It does not
require the velocity to point toward the center. In the discrete
analysis, a center \(p\in C\) has an additional advantage. Since the
metric projection is nonexpansive, \(\|\proj_C(y)-p\|\le\|y-p\|\) for
every \(y\in\R^n\); see \cite[Corollary~12.20]{Rock}. Hence the projection
step never moves a predicted point farther from \(p\). Small time steps
are still needed, since~\eqref{A3} controls the radial component of the
drift but not its size.

The following lemma extends \eqref{A3} to the whole set \(C\) and gives
the constant \(M_p\) in \eqref{eq:A3-anchor} explicitly.

\begin{lemma}[Globalized dissipativity]\label{lem:A3-global}
Assume \eqref{A1}, \eqref{ass:G}, and \eqref{A3} at a center
\(p\in\R^n\). Set \(K_\star:=C\cap\overline B(p;R_\star)\) and define
\[
M_{R_\star}:=
\sup\bigl(\{\|w\|:x\in K_\star,\ w\in F(x)\}\cup\{0\}\bigr),
\qquad
\widetilde M:=\max\{M,R_\star M_{R_\star}+\gamma R_\star^2\}.
\]
Then \(M_{R_\star}<\infty\), and
\begin{equation}\label{eq:A3_global}
\sup_{v\in\proj_{T_C(x)}F(x)}\langle x-p,v\rangle
\le\widetilde M-\gamma\|x-p\|^2,
\qquad x\in C.
\end{equation}
In particular, \(\widetilde M\) is a global dissipativity constant at
\(p\) in~\eqref{eq:A3-anchor}.
\end{lemma}

\begin{proof}
We split the proof into two steps.

\textbf{Step~1}. \emph{The inner region.}
Compactness of \(K_\star\), local boundedness of \(F\), and a finite
covering give \(M_{R_\star}<\infty\). The union with \(\{0\}\) in the
definition of \(M_{R_\star}\) covers the case \(K_\star=\emptyset\).

\textbf{Step~2}. \emph{The estimate.}
If \(\|x-p\|\ge R_\star\), we use \eqref{A3} and \(\widetilde M\ge M\).
If \(\|x-p\|\le R_\star\) and \(w\in F(x)\), then the projection onto
\(T_C(x)\) does not increase the norm. Hence
\[
\langle x-p,\proj_{T_C(x)}w\rangle
\le\|x-p\|\,\|w\|
\le R_\star M_{R_\star}
\le\widetilde M-\gamma\|x-p\|^2.
\]
Taking the supremum proves \eqref{eq:A3_global}.
\end{proof}

Under \eqref{A2}, we have \(M_{R_\star}\le a+b(\|p\|+R_\star)\). Thus
\(\max\{M,R_\star(a+b\|p\|+bR_\star)+\gamma R_\star^2\}\) is also an
admissible choice of \(\widetilde M\). If \(p=0\), we write \((A3_0)\) for
\eqref{eq:A3-anchor} with \(M_0=\widetilde M\).

When \eqref{A2} and \eqref{A3} are imposed together on an unbounded set
\(C\), their constants cannot be chosen independently.

\begin{remark}[Compatibility of growth and dissipation]
\label{rem:growth-dissipation}
If \(C\) is unbounded, then \eqref{A2} and \eqref{A3} at \(p=0\) imply
\(\gamma\le b\). Indeed, for \(x\in C\) with \(r:=\|x\|\ge R_\star\) and
\(v\in\proj_{T_C(x)}F(x)\), we have
\(-r(a+br)\le\langle x,v\rangle\le M-\gamma r^2\). Hence
\((\gamma-b)r^2\le M+ar\). Dividing by \(r^2\) along an unbounded sequence
in \(C\), we obtain \(\gamma\le b\). This restriction is compatible with
the uniform-in-time discrete bounds of Section~\ref{sec:energy}, which
hold under suitable conditions on the step sizes and the projection
errors.
\end{remark}

\section{Dissipative energy bounds: continuous and discrete levels}
\label{sec:energy}

In this section we derive a priori bounds for the solutions
of~\eqref{eq:F-formulation} and for the iterates of the catching-up scheme.
They provide the bounded regions whose role was explained in
Section~\ref{sec:setting},
and they do not require \(C\) to be compact. Our main tool is the quadratic
energy about a center \(p\). For the continuous dynamics, the tangent
dissipativity~\eqref{A3} gives a differential inequality for this energy,
without~\eqref{A2}. For the discrete dynamics, each step adds a predictor term and a projection
error to the energy. We control these terms with~\eqref{A3} at a center
\(p\in C\) and the nonexpansiveness of the projection, without~\eqref{A2},
which yields confinement for small steps and small projection errors.
Under~\eqref{A2}, for a fixed initial state and a bounded accumulated error
budget, a simpler argument bounds the nodes on every finite horizon; see
Corollary~\ref{cor:limit-linear-growth}.

\subsection{Continuous energy estimates}

We start with the continuous dynamics. The bounds below hold for every
solution, do not require~\eqref{A2}, and serve as the model for the
discrete estimates. As shown in Section~\ref{sec:setting}, every solution
on an interval \(I\) satisfies the projected formulation
\begin{equation}\label{eq:proj-dynamics}
\dot x(t)\in\proj_{T_C(x(t))}F(x(t)),
\qquad \text{a.e. }t\in I.
\end{equation}
Taking the inner product with \(x(t)-p\) and using the global
form~\eqref{eq:A3-anchor}, we arrive at the following energy estimate.

\begin{proposition}[Continuous energy estimate and global bound]
\label{prop:cont_energy}
Fix \(p\in\R^n\). Assume \eqref{A1}, \eqref{ass:G}, and \eqref{A3} at \(p\).
Fix the rate \(\gamma\) and a global dissipativity constant \(M_p\)
in \eqref{eq:A3-anchor}; such a constant exists by
Lemma~\ref{lem:A3-global}. Let \(I=[0,T_\ast)\), where
\(0<T_\ast\le\infty\), and let \(x:I\to C\) be a locally absolutely
continuous solution of \eqref{eq:F-formulation}.
Set \(V_p(x):=\frac12\|x-p\|^2\) and, for \(t\ge0\), define
\begin{equation}\label{eq:beta-def}
\rev{\beta_p}(t):=
e^{-2\gamma t}\|x_0-p\|^2
+\frac{M_p}{\gamma}\bigl(1-e^{-2\gamma t}\bigr).
\end{equation}
Then the following assertions hold.
\par\smallskip\noindent\textnormal{(i)}
For almost every \(t\in I\),
\begin{equation}\label{eq:Vdot-estimate}
\frac{d}{dt}V_p(x(t))
\le M_p-\gamma\|x(t)-p\|^2.
\end{equation}
\par\smallskip\noindent\textnormal{(ii)}
For every \(t\in I\),
\begin{equation}\label{eq:cont-energy-bound}
\|x(t)-p\|^2\le\rev{\beta_p}(t).
\end{equation}
In particular,
\(\sup_{t\in I}\|x(t)-p\|^2
\le\max\{\|x_0-p\|^2,M_p/\gamma\}\).
\end{proposition}

\begin{proof}
We split the proof into two steps.

\textbf{Step~1}. \emph{Differential inequality.}
On every compact subinterval of \(I\), \(V_p\circ x\) is absolutely
continuous. By the chain rule, \eqref{eq:proj-dynamics} and
the global form \eqref{eq:A3-anchor},
\(\frac{d}{dt}V_p(x(t))
=\langle x(t)-p,\dot x(t)\rangle
\le M_p-\gamma\|x(t)-p\|^2\)
for almost every \(t\in I\). This proves
\eqref{eq:Vdot-estimate}.

\textbf{Step~2}. \emph{Integration.}
Since \(\|x(t)-p\|^2=2V_p(x(t))\), multiplication by
\(e^{2\gamma t}\) gives
\(\frac{d}{dt}\bigl(e^{2\gamma t}V_p(x(t))\bigr)
\le M_p e^{2\gamma t}\) almost everywhere.
For any \(t\in I\), integration over \([0,t]\), followed by
multiplication by \(2e^{-2\gamma t}\), yields
\eqref{eq:cont-energy-bound}.
\end{proof}

Proposition~\ref{prop:cont_energy} is an a priori bound: it controls every
solution that exists on \(I\), but it does not establish existence or
global continuation. These questions are settled in
Theorem~\ref{thm:global-wellposed}.

\subsection{The catching-up scheme}
\label{subsec:discrete}

We now turn to the discrete dynamics generated by the catching-up scheme.
At each step, an explicit Euler step along a selection of \(F\) is
followed by an approximate projection onto \(C\), which allows an error
at each step. We measure these errors by the budgets introduced below,
define the scheme on finite-horizon meshes and along a single infinite
sequence, and then compare it with the full-resolvent discretization.

\begin{definition}[Catching-up scheme: predictor--projection]
\label{def:scheme_energy}
\emph{Finite-horizon mesh families.}
Fix \(T>0\). For each \(\nu\in\mathbb N\), let
\[
\mathcal P_\nu:\quad
0=t_{\nu,0}<\cdots<t_{\nu,N_\nu}=T,
\]
with
\[
h_{\nu,k}:=t_{\nu,k+1}-t_{\nu,k},
\qquad
h_\nu:=\max_{0\le k<N_\nu}h_{\nu,k}.
\]
Attach tolerances \(\varepsilon_{\nu,k}\ge0\) to the intervals
and set
\[
Q_\nu(T):=
\max_{0\le k<N_\nu}
\frac{\varepsilon_{\nu,k}}{h_{\nu,k}^2},
\qquad
E_\nu(T):=
\sum_{k=0}^{N_\nu-1}
\frac{\varepsilon_{\nu,k}}{h_{\nu,k}}.
\]
Since \(\sum_{k=0}^{N_\nu-1}h_{\nu,k}=T\), we have
\[
E_\nu(T)
=\sum_{k=0}^{N_\nu-1}
h_{\nu,k}\frac{\varepsilon_{\nu,k}}{h_{\nu,k}^2}
\le TQ_\nu(T).
\]
We write \(\overline Q_T:=\sup_\nu Q_\nu(T)\), and refinement means
\(h_\nu\to0\). The quantities \(Q_\nu(T)\) and \(E_\nu(T)\) are the
pointwise and the accumulated error budgets, respectively.
For one row, we suppress \(\nu\) and write
\[
t_k=t_{\nu,k},\qquad
h_k:=h_{\nu,k},\qquad
\varepsilon_k=\varepsilon_{\nu,k},\qquad
x_k=x_{\nu,k}.
\]
Unless stated otherwise, the initial states satisfy
\(x_{\nu,0}=x_0\in C\).
\par\smallskip\noindent\emph{Separate long-horizon setting.}
For an individual infinite trajectory, the data satisfy
\[
h_k>0,\qquad
h_k\downarrow0,\qquad
\sum_{k=0}^{\infty}h_k=\infty,\qquad
\varepsilon_k\ge0,\qquad
\frac{\varepsilon_k}{h_k^2}\to0.
\]
These conditions concern a single infinite sequence. They are not
imposed on the finite rows above, and they do not describe a mesh
refinement.
\par\smallskip\noindent\emph{Predictor--projection update.}
In either setting, starting from the prescribed initial
state in \(C\), define recursively
\[
\begin{aligned}
y_{k+1}&=x_k+h_k w_k,
&\qquad w_k&\in F(x_k),\\
x_{k+1}&\in\proj_C^{\varepsilon_k}(y_{k+1}).
\end{aligned}
\]
Here \(\proj_C^{\varepsilon_k}\) is the approximate projection
defined in Section~\ref{sec:prelim}. The first relation is the
\emph{prediction step}; the second one is the \emph{correction step}, and
it enforces feasibility. Since \(F(x)\neq\emptyset\) for every \(x\in C\)
and approximate projections onto \(C\) are nonempty, both choices are
possible at every step. Consequently, all iterates remain in \(C\).
\end{definition}

The following algorithm summarizes the finite-horizon update of
Definition~\ref{def:scheme_energy}.
\begin{algorithm}[H]
\caption{Catching-up scheme with approximate projections}
\label{alg:catching-up}
\begin{algorithmic}[1]
\Require Initial state $x_0\in C$; grid $0=t_0<\cdots<t_N=T$; tolerances $\varepsilon_k\ge0$.
\Ensure Feasible nodes $(x_k)_{k=0}^{N}$.
\For{$k=0,\ldots,N-1$}
    \State $h_k\gets t_{k+1}-t_k$
    \State Choose $w_k\in F(x_k)$
    \State $y_{k+1}\gets x_k+h_k w_k$
    \State Choose $x_{k+1}\in\proj_C^{\varepsilon_k}(y_{k+1})$
\EndFor
\State \Return $(x_k)_{k=0}^{N}$
\end{algorithmic}
\end{algorithm}

The scheme of Definition~\ref{def:scheme_energy} has a forward--backward
structure, which we now make explicit. For $h>0$, write
$J_{h A}:=(\mathrm{Id}+h A)^{-1}$ for the resolvent of the maximally
monotone operator $A$, where $\mathrm{Id}$ denotes the identity on $\R^n$.
By Minty's theorem~\cite{Minty1962}, $J_{h A}$ is single-valued and defined
on all of $\R^n$; moreover, it is firmly nonexpansive; see
\cite[Chapter~23]{Bauschke}. For the normal cone, the corresponding
resolvent satisfies \(J_{h N_C}=(\mathrm{Id}+h N_C)^{-1}=\proj_C\).
The exact-projection update therefore has the forward--backward form \(x_{k+1}=J_{h_kN_C}\bigl(x_k+h_k(f(x_k)-g_k)\bigr)\), with \(g_k\in G(x_k)\).
Thus $f$ and the selected value of $G$ are evaluated explicitly
at $x_k$, while $N_C$ is treated through its resolvent.
In the approximate-projection scheme, the selection
$g_k\in G(x_k)$ remains exact, while the metric projection
is replaced by the feasible approximate projection of
Section~\ref{sec:prelim}.

For comparison, let $(x_k^{\mathrm{res}})$ denote the iterates
of the full-resolvent discretization, where the superscript
$\mathrm{res}$ stands for ``resolvent''. With the same initial state and
step sizes, these iterates satisfy
\[
x_{k+1}^{\mathrm{res}}
=J_{h_kA}\bigl(
x_k^{\mathrm{res}}+h_k f(x_k^{\mathrm{res}})
\bigr),
\qquad x_0^{\mathrm{res}}=x_0.
\]
This discretization treats the whole operator $A=G+N_C$
implicitly, while $f$ is evaluated explicitly.

\subsection{Discrete confinement without linear growth}
\label{sec:confinement-noA2}

This subsection establishes the discrete counterpart of
Proposition~\ref{prop:cont_energy} at a center \(p\in C\). Its confinement
theorems require neither~\eqref{A2} nor a bound depending on \(T\).
Lemma~\ref{lem:raw-dissipation} first transfers the dissipativity estimate
from the projected velocities to the drift itself.
Theorem~\ref{thm:confinement-anchor} then compares the shadow point
\(\widehat x_{k+1}=\proj_C(y_{k+1})\) with \(p\) through the
nonexpansiveness of the projection. In this way, \(\|w_k\|^2\) appears only
with the factor \(h_k^2\), and small steps preserve the dissipation.
Corollary~\ref{cor:confinement-origin} takes the origin as the center when
\(0\in C\), and Theorem~\ref{thm:confinement-HR} covers the case
\(0\notin C\).

The next remark explains why we do not estimate the energy at the nodes
directly.

\begin{remark}[Direct energy estimates at the nodes]
\label{rem:direct-energy}
{Under~\eqref{A2}, with a fixed initial state and
\(\sup_\nu E_\nu(T)<\infty\), a direct estimate at the nodes bounds them on
every finite horizon, but these bounds \rev{may} grow with \(T\). One may try to
obtain bounds independent of \(T\) from~\eqref{A3} through a direct energy
estimate. Expanding \(\|x_{k+1}-p\|^2\) then produces a cross term between
\(x_k-p\) and the projection correction \(x_{k+1}-y_{k+1}\), which Young's
inequality turns into a first-order term \(h_k\|w_k\|^2\). This term can
dominate the dissipation coming from~\eqref{A3}. The comparison of the
shadow point with \(p\) described above avoids this loss.}
\end{remark}

The predictor \(y_{k+1}=x_k+h_kw_k\) moves along the unprojected drift
\(w_k\in F(x_k)\), whereas the global form~\eqref{eq:A3-anchor} controls
only its tangential part. The following lemma closes this gap.

\begin{lemma}[From projected to unprojected dissipativity]
\label{lem:raw-dissipation}
Assume~\eqref{A1} and the canonical
representation~\eqref{eq:A0}--\eqref{eq:decomp}. Let $p\in\R^n$,
$M_p\ge0$, and $\gamma>0$. Then the following assertions hold.
\par\smallskip\noindent\textnormal{(i)} If
\[
\sup_{v\in\proj_{T_C(x)}F(x)}\langle x-p,v\rangle
\le M_p-\gamma\|x-p\|^2,\qquad x\in C,
\]
then
\begin{equation*}\tag{\(R_p\)}
\label{eq:raw-dissipation-anchor}
\sup_{w\in F(x)}\langle x-p,w\rangle
\le M_p-\gamma\|x-p\|^2,
\qquad x\in C.
\end{equation*}
\par\smallskip\noindent\textnormal{(ii)} If $p\in C$,
then~\eqref{eq:raw-dissipation-anchor} implies~\eqref{eq:A3-anchor} with
the same constants.
\par\smallskip\noindent\textnormal{(iii)} If~\eqref{eq:A3_global} holds at
\(p=0\) with constants $\gamma$ and $\widetilde M$, then
\begin{equation*}\tag{\(R\)}\label{eq:raw-dissipation}
\sup_{w\in F(x)}\langle x,w\rangle
\le\widetilde M-\gamma\|x\|^2,
\qquad x\in C.
\end{equation*}
\end{lemma}

\begin{proof}
We split the proof into three steps.

\textbf{Step~1}. \emph{Passing from interior points to the
boundary.} Assume the inequality in~(i). Fix $x\in C$ and
$g\in A_0(x)$. By~\eqref{eq:A0}, there exist $x_j\in E$ and
$a_j\in A(x_j)$ such that $x_j\to x$ and $a_j\to g$. Since
$x_j\in E$, we have $A(x_j)=\{a_j\}$. Moreover,
$E\subset\operatorname{int}(\dom A)$. Hence $T_C(x_j)=\R^n$ and
$N_C(x_j)=\{0\}$, and~\eqref{eq:decomp} gives $G(x_j)=\{a_j\}$.
Thus the projection in~(i) is the identity at $x_j$, and \(\langle x_j-p,\,f(x_j)-a_j\rangle\le M_p-\gamma\|x_j-p\|^2\).
Passing to the limit as $j\to\infty$ and using~\eqref{A1}, we arrive
at \(\langle x-p,\,f(x)-g\rangle\le M_p-\gamma\|x-p\|^2\).
For fixed $x$, the left-hand side is affine and continuous in $g$.
Hence the set of all $g$ satisfying this estimate is closed and
convex. It contains $A_0(x)$, and therefore it contains
$G(x)=\clco A_0(x)$. Since $F(x)=f(x)-G(x)$, this
proves~\eqref{eq:raw-dissipation-anchor}.

\textbf{Step~2}. \emph{The converse for $p\in C$.}
Assume~\eqref{eq:raw-dissipation-anchor} and $p\in C$. Fix $x\in C$
and $w\in F(x)$. By Moreau's decomposition, $w=v+n$ with
$v=\proj_{T_C(x)}w$ and $n\in N_C(x)$. Since $p\in C$, the normal
cone inequality gives $\langle x-p,n\rangle\ge0$. Therefore, \(\langle x-p,v\rangle=\langle x-p,w\rangle-\langle x-p,n\rangle\le M_p-\gamma\|x-p\|^2\).
Taking the supremum over $w\in F(x)$ yields~\eqref{eq:A3-anchor}.

\textbf{Step~3}. \emph{The case of the origin.}
Inequality~\eqref{eq:A3_global} at \(p=0\) is the inequality in~(i) with
$p=0$ and $M_0=\widetilde M$. Since Step~1 does not use $p\in C$, it
gives~\eqref{eq:raw-dissipation}. This completes the proof.
\end{proof}

Observe that Step~1 uses the canonical construction of $G$: every element
of $G(x)$ is generated by values of $A$ at interior points, where
$T_C=\R^n$. Step~2 uses only $p\in C$. Thus, for $p\in C$, the global
inequalities~\eqref{eq:A3-anchor} and~\eqref{eq:raw-dissipation-anchor}
are equivalent with the same constants.

Without the canonical structure, assertion~(i) may fail. For
example, let $C=[0,\infty)$, $F(x)=\{-x\}$ for $x>0$, and
$F(0)=[-1,0]$. This mapping is upper semicontinuous with compact
convex values, and the inequality in~(i) holds with $p=1$, $\gamma=1$
and $M_p=1$. However, the choice $x=0$ and $w=-1$ gives
$\langle x-p,w\rangle=1>0=M_p-\gamma\|x-p\|^2$. Note that the value
$-1\in F(0)$ is not a limit of values of $F$ at interior points.

Under~\eqref{A3} at \(p=0\), Lemmas~\ref{lem:A3-global}
and~\ref{lem:raw-dissipation} give constants $\gamma>0$ and
$\widetilde M\ge0$ such that~\eqref{A3} holds at \(p=0\) with this
$\gamma$ and~\eqref{eq:raw-dissipation} holds. We call any pair
$(\gamma,\widetilde M)$ with these two properties an \emph{admissible
pair}. Note that $\widetilde M$ need not be the constant produced by
Lemma~\ref{lem:A3-global}.

Without~\eqref{A2}, the next theorem provides the compact set \(K\)
used in Theorems~\ref{thm:limit_solution} and~\ref{thm:trajectory-rate}.
It is the intersection of \(C\) with a ball about \(p\), and it depends
neither on the row nor on the horizon. Assertion~(i) is the discrete
counterpart of assertion~(ii), which restates
Proposition~\ref{prop:cont_energy} as a confinement result.

We first fix the constants. Let \(p\in C\), and assume \eqref{A1},
\eqref{ass:G}, and \eqref{A3} at \(p\) with rate \(\gamma\). By
Lemma~\ref{lem:A3-global}, there is a constant \(M_p\ge0\) for which
\eqref{eq:A3-anchor} holds. Let $(x_{\nu,0})\subset C$ be a family of
initial states with
$r_{0,p}:=\sup_{\nu\in\mathbb N}\|x_{\nu,0}-p\|<\infty$. Put
\begin{equation}\label{eq:anchor-thresholds}
\begin{aligned}
R&:=2\max\Bigl\{1,\,r_{0,p},\,\sqrt{M_p/\gamma}\Bigr\},
& K&:=C\cap\overline B(p;R),\\
d&:=\gamma R^2-M_p,
& B&:=\sup\{\|w\|:x\in C,\ \|x-p\|\le R,\ w\in F(x)\},\\
q_*&:=\frac{d}{4R},
& h_*&:=\min\left\{\frac1{2\gamma},
                 \frac{d}{2(B^2+q_*^2)}\right\}.
\end{aligned}
\end{equation}
Each constant in~\eqref{eq:anchor-thresholds} depends only on the data
and on the constants listed before it. The radius $R$ places every
initial state in the inner ball of radius $R/2$, and it gives
$d\ge\frac34\gamma R^2$. The number $d$ measures the inward pull at the
edge of the ball. Indeed, \eqref{eq:raw-dissipation-anchor} gives
$\langle x-p,w\rangle\le-d$ for $\|x-p\|=R$ and $w\in F(x)$. The number
$B$ bounds the drift on the ball, so one explicit step from a node in the
ball moves it by at most $h_kB$. The thresholds $q_*$ and $h_*$ bound the
projection errors and the steps. They are chosen so that the inward pull
wins over the overshoot of a finite step and over the projection error;
see Step~2 of the proof below.

\begin{theorem}[Confinement about a feasible center for a family of initial states]
\label{thm:confinement-anchor}
Assume the canonical representation
\eqref{eq:A0}--\eqref{eq:decomp}, together with
\eqref{A1} and \eqref{ass:G}. Fix \(p\in C\), and assume \eqref{A3} at
\(p\). Fix its rate \(\gamma>0\) and a global dissipativity constant
\(M_p\ge0\) in \eqref{eq:A3-anchor}; such a constant exists by
Lemma~\ref{lem:A3-global}. Let \((x_{\nu,0})\subset C\) satisfy
\(r_{0,p}<\infty\), and let $R$, $K$, $q_*$, and $h_*$ be given
by~\eqref{eq:anchor-thresholds}. Then $q_*>0$ and $h_*>0$, and the
following assertions hold.
\par\smallskip\noindent\textnormal{(i)} Every finite or infinite sequence
generated by the updates of Definition~\ref{def:scheme_energy},
initialized at one of the $x_{\nu,0}$, and satisfying $0<h_k\le h_*$ and
$0\le\varepsilon_k\le q_*^2h_k^2$ at every index has all its nodes
and affine segments in $K$. In particular, for any $T>0$, this holds
for every finite-horizon row with $h_\nu\le h_*$ and
$Q_\nu(T)\le q_*^2$.
\par\smallskip\noindent\textnormal{(ii)} Every locally absolutely
continuous solution of the differential inclusion
in~\eqref{eq:F-formulation}, initialized at one of the $x_{\nu,0}$, or
at $x_0$ when $x_{\nu,0}\to x_0\in C$, remains in $K$ on its existence
interval.
\par\smallskip\noindent
The constants $R,d,B,q_*,h_*$ are independent of the row, the number
of updates, and the time horizon $T$.
\end{theorem}

\begin{proof}
We split the proof into three steps. Step~1 shows that the thresholds
are positive, Step~2 proves~(i) by induction on $k$, and Step~3
proves~(ii).

\textbf{Step~1}. \emph{Constants and local boundedness.}
Lemma~\ref{lem:raw-dissipation} gives~\eqref{eq:raw-dissipation-anchor}.
Since $R\ge2\sqrt{M_p/\gamma}$, we have $M_p\le\gamma R^2/4$, and hence \(d=\gamma R^2-M_p\ge\tfrac34\gamma R^2>0\).
The set $C\cap\overline B(p;R)$ is compact, and
\eqref{A1} and~\eqref{ass:G} imply that $F$ is bounded there. Hence
$B<\infty$. It follows that $q_*>0$ and $h_*>0$.

\textbf{Step~2}. \emph{Discrete confinement.}
Consider a sequence generated by the updates of
Definition~\ref{def:scheme_energy}, initialized at one of the
$x_{\nu,0}$, with $0<h_k\le h_*$ and
$0\le\varepsilon_k\le q_*^2h_k^2$. Assume that $\|x_k-p\|\le R$.
Then $\|w_k\|\le B$. Set $\widehat x_{k+1}:=\proj_C(y_{k+1})$ and
$e_k:=x_{k+1}-\widehat x_{k+1}$. Since $x_{k+1}\in C$, the projection
characterization in Section~\ref{sec:prelim} gives
$\langle y_{k+1}-\widehat x_{k+1},\,x_{k+1}-\widehat x_{k+1}\rangle\le0$.
Expanding $\|x_{k+1}-y_{k+1}\|^2$ and using the approximate-projection
inequality, we arrive at
\[
\|e_k\|^2+d_C^2(y_{k+1})
\le\|x_{k+1}-y_{k+1}\|^2
\le d_C^2(y_{k+1})+\varepsilon_k .
\]
Thus $\|e_k\|\le\sqrt{\varepsilon_k}\le h_kq_*$.
Since $p\in C$, the nonexpansiveness of the metric projection
and~\eqref{eq:raw-dissipation-anchor} yield
\begin{align*}
\|\widehat x_{k+1}-p\|^2
&\le\|x_k-p+h_kw_k\|^2\\
&\le(1-2\gamma h_k)\|x_k-p\|^2+2M_ph_k+h_k^2B^2\\
&\le R^2-2h_kd+h_k^2B^2 ,
\end{align*}
where the last inequality uses $1-2\gamma h_k\ge0$,
$\|x_k-p\|\le R$, and $d=\gamma R^2-M_p$.
Moreover, $h_k\le h_*\le d/(2(B^2+q_*^2))$ gives
$h_k^2B^2\le h_kd/2$. Hence
$\|\widehat x_{k+1}-p\|^2\le R^2-\frac32h_kd<R^2$, that is,
$\|\widehat x_{k+1}-p\|<R$. Expanding $x_{k+1}=\widehat x_{k+1}+e_k$
and using the Cauchy--Schwarz inequality, $\|\widehat x_{k+1}-p\|\le R$,
and $\|e_k\|\le h_kq_*$, we obtain
\begin{align*}
\|x_{k+1}-p\|^2
&\le R^2\underbrace{-2h_kd}_{\text{inward pull}}
   \underbrace{+h_k^2(B^2+q_*^2)}_{\text{overshoot}}
   \underbrace{+2h_kq_*R}_{\text{projection error}}\\
&\le R^2-h_kd<R^2.
\end{align*}
Here $2h_kq_*R=h_kd/2$ and
$h_k^2(B^2+q_*^2)\le h_kd/2$ by the definitions of $q_*$ and
$h_*$. Since $\|x_{\nu,0}-p\|\le r_{0,p}\le R/2$, induction proves
that all nodes lie in $K$, and the convexity of $K$ gives the assertion
for every affine segment. For a finite-horizon row, we have
$h_k\le h_\nu$ and $\varepsilon_k\le Q_\nu(T)h_k^2$. Hence the claim
of~(i) for finite-horizon rows with $h_\nu\le h_*$ and
$Q_\nu(T)\le q_*^2$ follows as well.

\textbf{Step~3}. \emph{Continuous confinement.}
For a solution initialized at \(x_{\nu,0}\),
Proposition~\ref{prop:cont_energy} with center \(p\) and the definition
of \(R\) give, throughout its existence interval,
\[
\|x(t)-p\|^2
\le e^{-2\gamma t}r_{0,p}^2
+\frac{M_p}{\gamma}(1-e^{-2\gamma t})
\le\frac{R^2}{4}.
\]
The same estimate applies to \(x_0\) when \(x_{\nu,0}\to x_0\), since
\(\|x_0-p\|\le r_{0,p}\). Thus every such solution remains in \(K\) on
its existence interval. This proves~(ii) and completes the proof.
\end{proof}

Since $h_*>0$ and $q_*>0$, refinement $h_\nu\to0$ ensures
$h_\nu\le h_*$ for all large $\nu$. If, in addition,
$Q_\nu(T)\le q_*^2$ for all large $\nu$, for instance when
$Q_\nu(T)\to0$, then assertion~(i) of
Theorem~\ref{thm:confinement-anchor} applies to all large rows, and the
finitely many remaining rows can be discarded in the convergence
analysis; see Corollary~\ref{cor:limit-anchored}. Assertion~(ii), in
turn, is an a priori bound. It applies to every solution that exists,
and the existence of solutions is established in
Theorem~\ref{thm:global-wellposed}.

We now take the origin as the center in
Theorem~\ref{thm:confinement-anchor}, which requires $0\in C$. Both
this case and the case $0\notin C$ below use bounds of the drift on
balls about the origin. For $r\ge0$, define
\[
B_r:=\sup\!\left(
\{\|w\|:x\in C,\ \|x\|\le r,\ w\in F(x)\}\cup\{0\}\right).
\]
This number is finite under~\eqref{A1} and~\eqref{ass:G}. With
$p=0$, the constant $B$ of Theorem~\ref{thm:confinement-anchor}
equals $B_R$.

\begin{corollary}[Discrete confinement when \(0\in C\)]
\label{cor:confinement-origin}
Assume the canonical representation~\eqref{eq:A0}--\eqref{eq:decomp},
together with~\eqref{A1}, \eqref{ass:G}, and~\eqref{A3} at \(p=0\).
Suppose that $0\in C$, fix $x_0\in C$, and choose an admissible pair
$(\gamma,\widetilde M)$. Let $R_{\mathrm{org}}$, $d_{\mathrm{org}}$,
$B_{\mathrm{org}}$, $q_{\mathrm{org}}$, and $h_{\mathrm{org}}$ be the
constants $R$, $d$, $B$, $q_*$, and $h_*$ of~\eqref{eq:anchor-thresholds}
for $p=0$, $M_p=\widetilde M$, and $x_{\nu,0}\equiv x_0$. Thus
$R_{\mathrm{org}}=2\max\{1,\|x_0\|,\sqrt{\widetilde M/\gamma}\}$ and
$B_{\mathrm{org}}=B_{R_{\mathrm{org}}}$. Then every finite or
infinite sequence generated by the updates of
Definition~\ref{def:scheme_energy}, initialized at $x_0$, and
satisfying $0<h_k\le h_{\mathrm{org}}$ and
$0\le\varepsilon_k\le q_{\mathrm{org}}^2h_k^2$ at every index has all
its nodes and affine segments in $C\cap\overline B(0;R_{\mathrm{org}})$.
\end{corollary}

\begin{proof}
For the chosen admissible pair, \eqref{eq:raw-dissipation} holds by
definition. Since \(0\in C\), Lemma~\ref{lem:raw-dissipation}
gives \eqref{eq:A3-anchor} with \(p=0\) and \(M_p=\widetilde M\).
Hence Theorem~\ref{thm:confinement-anchor} applies to the constant
initial family \(x_{\nu,0}\equiv x_0\) with the stated constants, and its
assertion~(i) places all nodes and affine segments in
\(C\cap\overline B(0;R_{\mathrm{org}})\).
\end{proof}

Thus, when $0\in C$, hypothesis~\eqref{A3} at \(p=0\) gives confinement
about the origin, without any compatibility condition, under the step
and tolerance thresholds of Corollary~\ref{cor:confinement-origin}.
Indeed, with $M_p=\widetilde M$, condition~\eqref{eq:A3-anchor} at $p=0$
is the global form~\eqref{eq:A3_global} of Lemma~\ref{lem:A3-global}.
If $0\notin C$, the origin cannot serve as the feasible center of
Theorem~\ref{thm:confinement-anchor}. Theorem~\ref{thm:confinement-HR}
below treats this case. It keeps the dissipativity about the origin
in~\eqref{A3} at \(p=0\), and it requires a quantitative compatibility
condition between the dissipativity and the distance from the origin
to $C$.

When $0\notin C$, the projection need not decrease the Euclidean norm,
so the normal part of a step may push the nodes away from the origin.
We control this effect by the distance $\delta_C:=d_C(0)$ and the
bounds $B_r$. For an admissible pair $(\gamma,\widetilde M)$ and a
radius $R_{\mathrm{HR}}>0$, consider the compatibility condition
\begin{equation*}\tag{\(H_{R_{\mathrm{HR}}}\)}
\label{eq:HR}
\gamma R_{\mathrm{HR}}^2>\widetilde M+\delta_C B_{R_{\mathrm{HR}}},
\end{equation*}
and put
\begin{equation}\label{eq:HR-thresholds}
\begin{aligned}
d_{\mathrm{HR}}&:=2\bigl(\gamma R_{\mathrm{HR}}^2-\widetilde M
          -\delta_C B_{R_{\mathrm{HR}}}\bigr),
\qquad
q_{\mathrm{HR}}:=\frac{d_{\mathrm{HR}}}{8(R_{\mathrm{HR}}+1)},\\
h_{\mathrm{HR}}&:=\min\left\{\frac1{2\gamma},
\frac1{2B_{R_{\mathrm{HR}}}+1},
\frac{d_{\mathrm{HR}}}{4(8B_{R_{\mathrm{HR}}}^2+q_{\mathrm{HR}}^2)}
\right\}.
\end{aligned}
\end{equation}
The number $d_{\mathrm{HR}}$ is twice the inward pull at the edge of
the ball, after the loss $\delta_C B_{R_{\mathrm{HR}}}$ caused by the normal
part. The thresholds $q_{\mathrm{HR}}$ and $h_{\mathrm{HR}}$ play the
same roles as $q_*$ and $h_*$ in~\eqref{eq:anchor-thresholds}. With these
constants, the next theorem confines the nodes to a ball centered at the
origin.

\begin{theorem}[Discrete confinement without linear growth]
\label{thm:confinement-HR}
Assume~\eqref{A1}, \eqref{ass:G}, and~\eqref{A3} at \(p=0\). Fix
$x_0\in C$, put $p_0:=\proj_C(0)$ and $\delta_C:=\|p_0\|=d_C(0)$.
Suppose that there exist an admissible pair $(\gamma,\widetilde M)$ and
a radius $R_{\mathrm{HR}}\ge\|x_0\|$ satisfying~\eqref{eq:HR}, and let
$d_{\mathrm{HR}}$, $q_{\mathrm{HR}}$, and $h_{\mathrm{HR}}$ be given
by~\eqref{eq:HR-thresholds}. Then these numbers are positive. Moreover,
every finite or infinite sequence generated by the updates of
Definition~\ref{def:scheme_energy}, initialized at $x_0$, and satisfying
$0<h_k\le h_{\mathrm{HR}}$ and
$0\le\varepsilon_k\le q_{\mathrm{HR}}^2h_k^2$ at every index
satisfies
\[
\|x_k\|\le R_{\mathrm{HR}},
\qquad
\|x_{k+1}\|^2\le R_{\mathrm{HR}}^2-\frac{d_{\mathrm{HR}}}{2}h_k .
\]
Its affine segments remain in $C\cap\overline B(0;R_{\mathrm{HR}})$.
\end{theorem}

\begin{proof}
By \eqref{eq:HR}, \(d_{\mathrm{HR}}>0\). Since \(B_{R_{\mathrm{HR}}}<\infty\),
the constants \(q_{\mathrm{HR}}\) and \(h_{\mathrm{HR}}\) are positive.
It remains to prove the estimates for the nodes by induction. We split
the induction step into three steps.

\textbf{Step~1}. \emph{Projection and normal estimates.}
Assume that \(\|x_k\|\le R_{\mathrm{HR}}\), and write
\(h:=h_k\) and \(w:=w_k\). Set
\[
y:=x_k+h w,\qquad
\widehat x:=\proj_C(y),\qquad
\eta:=\frac{y-\widehat x}{h},\qquad
e:=x_{k+1}-\widehat x.
\]
The projection characterization gives \(\eta\in N_C(\widehat x)\).
Moreover, \(\|w\|\le B_{R_{\mathrm{HR}}}\).
Since \(x_k\in C\) and \(\widehat x-x_k=h(w-\eta)\),
\[
\|\eta\|=\frac{d_C(y)}{h}
\le\|w\|\le B_{R_{\mathrm{HR}}},
\qquad
\|\widehat x-x_k\|\le2h B_{R_{\mathrm{HR}}}.
\]
Testing the normal inequality at \(p_0=\proj_C(0)\in C\) gives
\(\langle\widehat x,\eta\rangle
\ge\langle p_0,\eta\rangle\ge-\delta_C B_{R_{\mathrm{HR}}}\).
Consequently,
\[
-\langle x_k,\eta\rangle
=-\langle\widehat x,\eta\rangle
+\langle\widehat x-x_k,\eta\rangle
\le\delta_C B_{R_{\mathrm{HR}}}
+2h B_{R_{\mathrm{HR}}}^2.
\]

\textbf{Step~2}. \emph{Energy estimate for the exact projection.}
Using \eqref{eq:raw-dissipation} and
\(\|w-\eta\|\le2B_{R_{\mathrm{HR}}}\), we obtain
\begin{align*}
\|\widehat x\|^2
&=\|x_k\|^2+2h\langle x_k,w\rangle
-2h\langle x_k,\eta\rangle+h^2\|w-\eta\|^2\\
&\le(1-2\gamma h)\|x_k\|^2
+2h(\widetilde M+\delta_C B_{R_{\mathrm{HR}}})
+8h^2B_{R_{\mathrm{HR}}}^2\\
&\le R_{\mathrm{HR}}^2-h d_{\mathrm{HR}}
+8h^2B_{R_{\mathrm{HR}}}^2.
\end{align*}
The last inequality uses \(h\le1/(2\gamma)\) and the definition of
\(d_{\mathrm{HR}}\).

\textbf{Step~3}. \emph{Projection error and induction.}
As in the proof of Theorem~\ref{thm:confinement-anchor},
\(\|e\|\le\sqrt{\varepsilon_k}\le h q_{\mathrm{HR}}\). Also,
\(\|\widehat x\|\le R_{\mathrm{HR}}
+2h B_{R_{\mathrm{HR}}}\le R_{\mathrm{HR}}+1\),
since \(h\le(2B_{R_{\mathrm{HR}}}+1)^{-1}\). Therefore
\begin{align*}
\|x_{k+1}\|^2
&\le R_{\mathrm{HR}}^2-h d_{\mathrm{HR}}
+h^2(8B_{R_{\mathrm{HR}}}^2+q_{\mathrm{HR}}^2)
+2h q_{\mathrm{HR}}(R_{\mathrm{HR}}+1)\\
&\le R_{\mathrm{HR}}^2
-\frac{d_{\mathrm{HR}}}{2}h
<R_{\mathrm{HR}}^2.
\end{align*}
Indeed, the definitions of \(q_{\mathrm{HR}}\) and \(h_{\mathrm{HR}}\)
give \(2h q_{\mathrm{HR}}(R_{\mathrm{HR}}+1)=h d_{\mathrm{HR}}/4\) and
\(h^2(8B_{R_{\mathrm{HR}}}^2+q_{\mathrm{HR}}^2)\le h d_{\mathrm{HR}}/4\).
Since \(\|x_0\|\le R_{\mathrm{HR}}\), induction gives the claimed
estimates at every update. Since \(C\cap\overline B(0;R_{\mathrm{HR}})\)
is convex, the affine segments of the sequence also remain in this set.
This completes the proof.
\end{proof}

Under~\eqref{A2}, condition~\eqref{eq:HR} imposes no additional
restriction. Indeed, fix an admissible pair $(\gamma,\widetilde M)$.
Then $B_{R_{\mathrm{HR}}}\le a+bR_{\mathrm{HR}}$, so the right-hand side
of~\eqref{eq:HR} is at most $\widetilde M+d_C(0)(a+bR_{\mathrm{HR}})$,
which grows linearly in $R_{\mathrm{HR}}$, whereas the left-hand side
$\gamma R_{\mathrm{HR}}^2$ is quadratic. Hence~\eqref{eq:HR} holds for
all sufficiently large radii $R_{\mathrm{HR}}\ge\|x_0\|$.
Without~\eqref{A2}, it is a genuine hypothesis of
Theorem~\ref{thm:confinement-HR}.

Theorem~\ref{thm:confinement-anchor} uses dissipativity about a center
in $C$, whereas Theorem~\ref{thm:confinement-HR} uses dissipativity
about the origin. The next remark shows that, without~\eqref{A2}, the
choice of the center is not a matter of convenience.

\begin{remark}[Dependence on the center]
\label{rem:regime-comparison}\label{rem:A3-anchor-comparison}
Assume first~\eqref{A2}, and let~\eqref{eq:A3_global} hold at \(p=0\) with
constants \((\gamma,\widetilde M)\). Fix \(p\in\R^n\). For \(x\in C\) and
\(v\in\proj_{T_C(x)}F(x)\), we have \(\|v\|\le a+b\|x\|\), and Young's
inequality gives
\(\langle x-p,v\rangle\le\widetilde M+a\|p\|+b^2\|p\|^2/(2\gamma)
-\frac{\gamma}{2}\|x\|^2\). Since \(\|x-p\|^2\le2\|x\|^2+2\|p\|^2\), the
global form~\eqref{eq:A3-anchor} holds at \(p\) with \(\gamma_p=\gamma/4\)
and \(M_p=\widetilde M+a\|p\|+b^2\|p\|^2/(2\gamma)+\gamma\|p\|^2/2\). The
same argument, with \(\|v\|\le(a+b\|p\|)+b\|x-p\|\), gives the converse
implication from the center \(p\) back to the origin.
Without~\eqref{A2}, the center can no longer be changed in this way. Let
\(C=\R^2\), \(A=0\), \(p=(1,0)\), \(J(x_1,x_2)=(-x_2,x_1)\), and consider
the fields \(F_0(x)=-x+\|x\|^2Jx\) and \(F_p(x)=F_0(x-p)\). With \(f=F_0\)
or \(f=F_p\), both fields satisfy~\eqref{A1} and~\eqref{ass:G}, and both
violate~\eqref{A2}. Since \(\langle z,Jz\rangle=0\), we have
\(\langle x,F_0(x)\rangle=-\|x\|^2\) and
\(\langle x-p,F_p(x)\rangle=-\|x-p\|^2\), so \(F_0\) satisfies~\eqref{A3}
at \(0\) and \(F_p\) satisfies~\eqref{A3} at \(p\). However,
\(\langle(0,t)-p,F_0(0,t)\rangle=t^3-t^2=\langle(1,-t),F_p(1,-t)\rangle\).
Hence \(F_0\) violates~\eqref{A3} at \(p\), and \(F_p\) violates~\eqref{A3}
at \(0\).
\end{remark}

We conclude this section with an example that applies
Corollary~\ref{cor:confinement-origin} to a cubic drift with an active
boundary constraint. This drift violates~\eqref{A2}, and its explicit
solution also serves as a test case for the trajectory-error estimates of
Section~\ref{sec:quantitative}.

\begin{example}[Confinement for a field without linear growth]
\label{ex:cubic}
Let \(C=\mathbb R\times[0,\infty)\), \(A=N_C\), and \(f(u,v)=(-u^3,-1-v)\).
The interior values of \(A\) vanish. Thus \(A_0=G=\{0\}\), \(F=f\), and
\eqref{A1} and~\eqref{ass:G} hold. The tangential projection changes
\(F\) only at the boundary points \(v=0\), where it replaces the second
component by zero. Since \(-u^4+u^2\le1/4\),
\[
\left\langle (u,v),\proj_{T_C(u,v)}F(u,v)\right\rangle
\le\frac14-u^2-v^2.
\]
Hence~\eqref{A3} holds at \(p=0\) with \(R_\star=1\), \(\gamma=1\),
and \(M=1/4\). On \(C\cap\overline B(0;1)\), the bound
\(\|F\|^2\le4\) is attained at \((0,1)\), so \(M_{R_\star}=2\), and
Lemma~\ref{lem:A3-global} gives \(\widetilde M=3\).
For \(x_0=(1,1)\), Corollary~\ref{cor:confinement-origin} gives
\[
\begin{gathered}
R_{\mathrm{org}}=2\sqrt3,\qquad d_{\mathrm{org}}=9,\qquad
q_{\mathrm{org}}=\frac{3\sqrt3}{8},\\
B_{\mathrm{org}}\le\widehat B:=\sqrt{1741+4\sqrt3},\qquad
\widehat h_{\mathrm{org}}=\frac9{2(\widehat B^2+q_{\mathrm{org}}^2)}
\approx2.574\times10^{-3}.
\end{gathered}
\]
Hence the conditions \(0\le\varepsilon_k\le q_{\mathrm{org}}^2h_k^2\) and
\(0<h_k\le\widehat h_{\mathrm{org}}\le h_{\mathrm{org}}\) are sufficient
for the confinement in Corollary~\ref{cor:confinement-origin}.
The displayed dissipativity inequality also gives the sharper global
constant \(M_p=1/4\). With this choice,
Theorem~\ref{thm:confinement-anchor} gives \(R=2\sqrt2\) and \(K=C\cap\overline B(0;2\sqrt2)\).
For the error estimate of Section~\ref{sec:quantitative}, note that the
all-pairs condition holds with \(\ell_K=0\), since
\(\langle x-x',F(x)-F(x')\rangle=-(u-u')^2(u^2+uu'+u'^2)-(v-v')^2\le0\).
On the other hand, \(d(0,F(u,0))=\sqrt{u^6+1}\ge|u|^3\),
so~\eqref{A2} fails. The same bound violates the minimal-norm growth
hypothesis of~\cite{Vilches2024}. Indeed, a globally Lipschitz majorant
\(L\) with constant \(L_{\mathrm{Lip}}\) satisfies
\(L(u,0)\le L(0,0)+L_{\mathrm{Lip}}|u|\), so it cannot dominate this cubic
growth. Nevertheless, the confinement above holds.
For \(x_0=(1,1)\), the unique global solution supplied by
Corollary~\ref{cor:limit-anchored} and
Theorem~\ref{thm:global-wellposed}\textnormal{(ii)} is
\[
x(t)=\left(\frac1{\sqrt{1+2t}},\max\{2e^{-t}-1,0\}\right).
\]
It remains in \rev{\(C\cap\overline B(0;\sqrt2)\)}. 
 At \(t_\partial=\log2\), it reaches the boundary;
thereafter the normal reaction \((0,-1)\) cancels the second component
of \(F\). For \(T=1\) and exact projections,
Theorem~\ref{thm:trajectory-rate} applies on \(K\) with \(\ell_K=0\).
The exact updates \(u_{k+1}=u_k(1-hu_k^2)\) and
\(v_{k+1}=\max\{v_k-h(1+v_k),0\}\) preserve \([0,1]^2\) for \(0<h\le1\),
so the iterates stay in \(K\). In computations with
\(h=10^{-2},10^{-3},10^{-4}\), the sup-errors decay at first order,
faster than the proved rate \(\sqrt h\). The constants of
Theorem~\ref{thm:trajectory-rate} for this example and the computed
errors are given in the supplementary material.
\end{example}

\section{Convergence of the catching-up scheme}
\label{sec:convergence}

This section is devoted to the convergence of the catching-up scheme on a
fixed horizon \([0,T]\), using the compact convex sets supplied by the
bounds and the confinement results of Section~\ref{sec:energy}. In
Section~\ref{subsec:interpolants}, we embed the nodes into continuous time
and introduce the shadow points. Section~\ref{subsec:finite-horizon} shows
that, when the mesh size and the accumulated budget tend to zero and the
nodes stay in a common compact convex set, the interpolants converge along
a subsequence to a solution; it also provides such a set in the cases with
and without~\eqref{A2}. Section~\ref{subsec:whole-family} shows that a
locally Lipschitz perturbation gives uniqueness and the convergence of the
whole family. Finally, Section~\ref{subsec:global-wellposed} establishes
global existence, uniqueness and continuous dependence on the initial data.

\subsection{Interpolants and shadow points}
\label{subsec:interpolants}

The interpolants rewrite the discrete dynamics in differential form. Their
velocity splits into a component along \(F\), an exact normal component at
the shadow point, and a residual caused by the projection error. This
representation is the key tool for passing to the limit
in~\eqref{eq:F-formulation}.

Fix \(T>0\). We use the family of Definition~\ref{def:scheme_energy},
with \(h_\nu\), \(Q_\nu(T)\) and \(E_\nu(T)\) as defined there. The symbol
\(\mu\) labels a row and serves as a subscript for its interpolants; a
sequence of rows is denoted by \(\mu_j\). Within a row, we suppress \(\nu\)
and write \(k_T:=N_\nu\), \(h_k:=h_{\nu,k}\) and
\(\varepsilon_k:=\varepsilon_{\nu,k}\) for \(0\le k<k_T\). The grid
satisfies
\[
t_0=0,\qquad
t_k=\sum_{i=0}^{k-1}h_i
\quad (1\le k\le k_T),\qquad
t_{k_T}=T.
\]
Note that \(h_k\le h_\nu\) and \(\varepsilon_k\le Q_\nu(T)h_k^2\) for
\(0\le k<k_T\). For the separate infinite-sequence setting, we set
\(t_0:=0\), \(t_k:=\sum_{i=0}^{k-1}h_i\) for \(k\ge1\), and
\(k_T:=\max\{k\ge0:t_k\le T\}\).

\begin{definition}[Interpolants and discrete derivatives]
\label{def:interpolants}
Fix \(T>0\) and a finite row with the notation above; thus \(t_{k_T}=T\).
Let \((x_k)\) be generated by the catching-up scheme
\[
y_{k+1}=x_k+h_k w_k,\qquad w_k\in F(x_k),\qquad
x_{k+1}\in\proj_C^{\varepsilon_k}(y_{k+1}),
\qquad 0\le k<k_T.
\]
Define the exact shadow point, the scaled normal, and the residual by
\[
\widehat x_{k+1}:=\proj_C(y_{k+1}),\qquad
\eta_k:=
\frac{y_{k+1}-\widehat x_{k+1}}{h_k},\qquad
\rho_k:=\frac{x_{k+1}-\widehat x_{k+1}}{h_k}.
\]
\par\smallskip\noindent\emph{Signed projection correction.}
The signed projection correction is \(p_k:=x_{k+1}-y_{k+1}\), and we set
\(v_k:=p_k/h_k=\rho_k-\eta_k\).
\par\smallskip\noindent\emph{State and velocity interpolants.}
For \(0\le k<k_T\) and \(t\in[t_k,t_{k+1})\), define the piecewise
affine state interpolant and the piecewise constant terms by
\[
\begin{aligned}
x_\mu(t)
&:=x_k+\frac{t-t_k}{h_k}(x_{k+1}-x_k),\\
(w_\mu,v_\mu,\eta_\mu,\rho_\mu)(t)
&:=(w_k,v_k,\eta_k,\rho_k).
\end{aligned}
\]
Set \(x_\mu(T):=x_{k_T}\) and assign the value zero at \(T\) to
\(w_\mu,v_\mu,\eta_\mu,\rho_\mu\).
For the separate infinite-sequence setting, where \(t_k\to\infty\),
we use the same interpolation formulas on the grid intervals with
\(t_k<T\), restricted to \([0,T)\), and define \(x_\mu(T)\) by continuity.
The bounds involving \(Q_\nu(T)\) and \(E_\nu(T)\) concern the finite rows.
\end{definition}

We now derive two basic properties of the quantities in
Definition~\ref{def:interpolants}. The projection characterization gives
\(\eta_k\in N_C(\widehat x_{k+1})\). Testing its variational inequality at
\(x_{k+1}\in C\) and using the approximate projection inequality, we obtain
\begin{equation}\label{eq:shadow-residual}
\|x_{k+1}-\widehat x_{k+1}\|^2
\le
\|x_{k+1}-y_{k+1}\|^2-d_C(y_{k+1})^2
\le\varepsilon_k.
\end{equation}
Consequently, \(\|\rho_k\|\le\sqrt{\varepsilon_k}/h_k\), and for each
finite row,
\[
\|\rho_\mu\|_{L^\infty(0,T)}^2\le Q_\nu(T),
\qquad
\|\rho_\mu\|_{L^2(0,T)}^2
=\sum_{k=0}^{k_T-1}h_k\|\rho_k\|^2
\le E_\nu(T).
\]
Moreover, the state interpolant is absolutely continuous on \([0,T]\),
with \(x_\mu(0)=x_{\nu,0}\), and it satisfies \(\dot x_\mu=w_\mu+v_\mu=w_\mu-\eta_\mu+\rho_\mu\) almost everywhere on \([0,T]\).
In the infinite-sequence setting, only finitely many intervals are
involved. Hence \(x_\mu\) is again absolutely continuous on \([0,T]\), with
\(x_\mu(0)=x_0\) and the same velocity identity.

The next lemma collects estimates for one row of the scheme. They separate
the predictor distance from the error in the corrector.

\begin{lemma}[One-step predictor and corrector estimates]
\label{lem:proj-error}
Let \(C\subset\mathbb R^n\) be nonempty, closed, and convex. Fix
\(T>0\) and a finite row from Definition~\ref{def:scheme_energy}.
Suppose that all its nodes belong to a set \(K\subset C\) and that
\[
M_K:=\sup\{\|w\|:x\in K,\ w\in F(x)\}<\infty.
\]
With the quantities of Definition~\ref{def:interpolants}, the following
estimates hold for every \(0\le k<k_T\).
\par\smallskip\noindent\textnormal{(i)} The predictor and the exact shadow
point satisfy
\[
\|w_k\|\le M_K,\quad d_C(y_{k+1})\le M_Kh_k,\quad
\|\eta_k\|=\frac{d_C(y_{k+1})}{h_k}\le M_K,
\]
\[
\|\widehat x_{k+1}-x_k\|\le\|y_{k+1}-x_k\|\le M_Kh_k.
\]
\par\smallskip\noindent\textnormal{(ii)} The corrector and the residual
satisfy
\[
\|x_{k+1}-\widehat x_{k+1}\|^2\le\varepsilon_k,\qquad
\|p_k\|^2\le M_K^2h_k^2+\varepsilon_k,
\]
\[
\|\rho_\mu\|_{L^\infty(0,T)}^2\le Q_\nu(T),\qquad
\|\rho_\mu\|_{L^2(0,T)}^2\le E_\nu(T).
\]
For \(t\in[t_k,t_{k+1})\), the state interpolant satisfies
\[
\|x_\mu(t)-x_k\|,\ \|x_\mu(t)-\widehat x_{k+1}\|
\le h_k(M_K+\|\rho_k\|).
\]
Moreover,
\[
\sum_{k<N_\nu}\|p_k\|^2\le h_\nu(M_K^2T+E_\nu(T)),\qquad
\sum_{k<N_\nu}\|p_k\|^2\le(M_K^2+Q_\nu(T))T h_\nu.
\]
\end{lemma}

\begin{proof}
We first prove the predictor bounds in~(i), and then the corrector
bounds in~(ii), the interpolant bounds, and the row estimates.
Since \(x_k\in K\subset C\) and \(w_k\in F(x_k)\), the predictor satisfies
\(d_C(y_{k+1})\le\|y_{k+1}-x_k\|=h_k\|w_k\|\).
The definition of \(\eta_k\) gives its stated norm. The nonexpansiveness
of \(\proj_C\) and \(\proj_C(x_k)=x_k\) give the shadow estimate.
The projection characterization and the approximate projection inequality
give \(\|x_{k+1}-\widehat x_{k+1}\|^2\le\varepsilon_k\), as
in~\eqref{eq:shadow-residual}, and the same inequality gives
\(\|p_k\|^2\le d_C(y_{k+1})^2+\varepsilon_k\).
The definitions of \(Q_\nu(T)\) and \(E_\nu(T)\) now give the residual
bounds. For \(\theta=(t-t_k)/h_k\in[0,1)\), we have
\[
x_\mu(t)-x_k=\theta(\widehat x_{k+1}-x_k+h_k\rho_k),\quad
x_\mu(t)-\widehat x_{k+1}=(1-\theta)(x_k-\widehat x_{k+1})+\theta h_k\rho_k,
\]
and the shadow estimate yields both interpolant bounds.
Summing the bound on \(\|p_k\|^2\) and using
\(\sum_kh_k^2\le Th_\nu\) and
\(\sum_k\varepsilon_k\le h_\nu E_\nu(T)\le Th_\nu Q_\nu(T)\)
gives the row estimates.
\end{proof}

Thus the predictor lies within \(M_Kh_k\) of \(C\). The scaled correction
\(v_k=\rho_k-\eta_k\) contains the normal reaction and need not vanish,
whereas the budgets \(Q_\nu(T)\) and \(E_\nu(T)\) control the residual
\(\rho_\mu\).

Note that the center \(p\) in~\eqref{eq:A3-anchor} differs from the correction \(p_k\), and the set
\(E\) of Section~\ref{sec:prelim} differs from the budget \(E_\nu(T)\).

\subsection{Convergence on finite horizons}
\label{subsec:finite-horizon}

We first bound the state interpolants of Definition~\ref{def:interpolants}
and establish their equicontinuity. The same estimates control the
selections \(w_\mu\), the shadow normals \(\eta_\mu\), and the residuals
\(\rho_\mu\). Recall that, by~\eqref{A1} and~\eqref{ass:G}, the field
\(F=f-G\) is upper semicontinuous on \(C\) with nonempty compact convex
values.

\begin{lemma}[Boundedness and equicontinuity of the interpolants]
\label{lem:bounded_equicontinuous}
Let \(C\subset\mathbb{R}^n\) be nonempty, closed and convex, and let
\(F\) be upper semicontinuous on \(C\) with nonempty compact convex values.
Fix \(T>0\) and a family of rows from Definition~\ref{def:scheme_energy},
with the interpolants of Definition~\ref{def:interpolants}.
Suppose that a compact convex set \(K\subset C\) contains every node
\(x_{\nu,k}\) of every row, and that
\(\overline E_T:=\sup_\nu E_\nu(T)<\infty\).
Then \(M_K:=\sup\{\|w\|:x\in K,\ w\in F(x)\}\) is finite, and
\(x_\mu(t)\in K\) for every \(t\in[0,T]\).
Moreover, for almost every \(t\in[0,T]\) in row \(\nu\),
\[
\|w_\mu(t)\|\le M_K,\qquad
\|\eta_\mu(t)\|\le M_K,\qquad
\|\rho_\mu(t)\|\le\sqrt{Q_\nu(T)}.
\]
The residual satisfies \(\|\rho_\mu\|_{L^2(0,T)}^2\le E_\nu(T)\), and
\[
\begin{aligned}
\|x_\mu(t)-x_\mu(s)\|
&\le 2M_K|t-s|+\sqrt{E_\nu(T)}\,|t-s|^{1/2}\\
&\le 2M_K|t-s|+\sqrt{\overline E_T}\,|t-s|^{1/2},
\qquad s,t\in[0,T].
\end{aligned}
\]
Hence the family \(\{x_\mu\}\) is uniformly bounded and equicontinuous
on \([0,T]\). If, in addition, \(\overline Q_T<\infty\), then the
interpolants are uniformly Lipschitz with constant
\(L_K:=2M_K+\sqrt{\overline Q_T}\):
\[
\|x_\mu(t)-x_\mu(s)\|\le L_K|t-s|,
\qquad s,t\in[0,T].
\]
\end{lemma}

\begin{proof}
We first bound the drift and the interpolants, and then estimate their
modulus of continuity. For each \(z\in K\), the compactness of \(F(z)\)
allows us to choose \(b_z>0\) such that \(F(z)\subset B(0;b_z)\).
Upper semicontinuity on \(C\) provides a relative neighborhood \(U_z\)
of \(z\) such that \(F(x)\subset B(0;b_z)\) for every \(x\in U_z\).
A finite subcover of \(K\) therefore gives \(M_K<\infty\).

Let \(t\in[0,T]\). At a node, \(x_\mu(t)\in K\) by hypothesis.
Otherwise, for some \(k\), Definition~\ref{def:interpolants} gives
\[
x_\mu(t)
=x_k+\frac{t-t_k}{h_k}(x_{k+1}-x_k)
=(1-\theta)x_k+\theta x_{k+1},
\qquad
\theta:=\frac{t-t_k}{h_k}\in[0,1].
\]
The convexity of \(K\) therefore gives \(x_\mu(t)\in K\) for every
\(t\in[0,T]\). Lemma~\ref{lem:proj-error} bounds \(w_k\), \(\eta_k\) and
\(\rho_k\), and gives \(\|\rho_\mu\|_{L^2(0,T)}^2\le E_\nu(T)\).
For \(0\le s<t\le T\), the identity \(\dot x_\mu=w_\mu-\eta_\mu+\rho_\mu\)
and the Cauchy--Schwarz inequality imply
\[
\begin{aligned}
\|x_\mu(t)-x_\mu(s)\|
&\le 2M_K(t-s)+\int_s^t\|\rho_\mu(\tau)\|\,d\tau\\
&\le 2M_K(t-s)+\sqrt{E_\nu(T)}\,(t-s)^{1/2}.
\end{aligned}
\]
Since \(E_\nu(T)\le\overline E_T<\infty\), this gives the stated common
modulus of continuity. Uniform boundedness follows from
\(x_\mu([0,T])\subset K\) and the boundedness of \(K\).

Finally, suppose that \(\overline Q_T<\infty\). The identity
\(\dot x_\mu=w_\mu-\eta_\mu+\rho_\mu\) now yields
\(\|\dot x_\mu(t)\|\le L_K\) almost everywhere. Hence, for
\(0\le s<t\le T\), \(\|x_\mu(t)-x_\mu(s)\|\le\int_s^t\|\dot x_\mu(\tau)\|\,d\tau\le L_K|t-s|\).
\end{proof}

By Lemma~\ref{lem:bounded_equicontinuous} and the Arzel\`a--Ascoli
theorem, the family \(\{x_\mu\}\) is relatively compact in
\(C([0,T];\mathbb{R}^n)\). The selections \(w_\mu\) and the shadow normals
\(\eta_\mu\) are bounded in \(L^\infty(0,T;\mathbb{R}^n)\), and the
residuals \(\rho_\mu\) tend to zero in \(L^2(0,T;\mathbb{R}^n)\) when
\(E_\nu(T)\to0\). These facts allow us to pass to the limit in the
identity \(\dot{x}_\mu=w_\mu-\eta_\mu+\rho_\mu\). The next theorem
identifies the limit as a solution of~\eqref{eq:F-formulation}.

\begin{theorem}[Convergence to solutions]
\label{thm:limit_solution}
Let \(C\subset\mathbb{R}^n\) be nonempty, closed and convex, and let \(F\)
be upper semicontinuous on \(C\) with nonempty compact convex values.
Fix \(T>0\) and a family of rows from Definition~\ref{def:scheme_energy},
with the interpolants \(x_\mu\), \(w_\mu\), \(v_\mu\), \(\eta_\mu\), and
\(\rho_\mu\) of Definition~\ref{def:interpolants}.
Suppose that a compact convex set \(K\subset C\) contains every node of
every row, and that
\[
h_\nu\longrightarrow0,\qquad
E_\nu(T)\longrightarrow0,\qquad
x_{\nu,0}\longrightarrow x_0\in C
\qquad(\nu\to\infty).
\]
Then \(\overline E_T:=\sup_\nu E_\nu(T)\) is finite, and the family
\(\{x_\mu\}\) is relatively compact in \(C([0,T];\mathbb{R}^n)\).
Moreover, every sequence of rows \(\mu_j\) with \(h_{\nu_j}\to0\) has a
subsequence (not relabeled) with the following property: there exist
\(x\in AC([0,T];\mathbb{R}^n)\) and \(w,\eta\in L^\infty(0,T;\mathbb{R}^n)\)
such that
\[
x_{\mu_j}\to x
\quad\text{uniformly on }[0,T],
\qquad
w_{\mu_j}\rightharpoonup^\ast w,\quad
\eta_{\mu_j}\rightharpoonup^\ast\eta
\quad\text{in }L^\infty(0,T;\mathbb{R}^n),
\]
and, for almost every \(t\in[0,T]\),
\[
\dot x(t)=w(t)-\eta(t),\qquad
w(t)\in F(x(t)),\qquad
\eta(t)\in N_C(x(t)).
\]
Here \(x(t)\in K\) for all \(t\in[0,T]\) and \(x(0)=x_0\).
The residuals and the scaled corrections satisfy
\[
\rho_{\mu_j}\to0
\quad\text{strongly in }L^2(0,T;\mathbb{R}^n),
\qquad
v_{\mu_j}\rightharpoonup-\eta
\quad\text{in }L^2(0,T;\mathbb{R}^n).
\]
If \(\overline Q_T<\infty\), then also
\(v_{\mu_j}\rightharpoonup^\ast-\eta\) in \(L^\infty(0,T;\mathbb{R}^n)\).
If \(Q_\nu(T)\to0\), then also \(\rho_{\mu_j}\to0\) strongly in
\(L^\infty(0,T;\mathbb{R}^n)\).
Consequently, \(x\) is a solution of~\eqref{eq:F-formulation} on
\([0,T]\) with \(x(0)=x_0\).
\end{theorem}

\begin{proof}
We split the proof into four steps. Set \(M_K:=\sup\{\|w\|:x\in K,\ w\in F(x)\}\).
Consider any sequence of rows \(\mu_j\) with \(h_{\nu_j}\to0\), and denote
their indices by \(\nu_j\). Since every fixed row has a positive mesh size,
\(\nu_j\to\infty\). Consequently, \(E_{\nu_j}(T)\to0\).

\textbf{Step~1}. \emph{Compactness.}
Lemma~\ref{lem:bounded_equicontinuous} yields the common compact range
\(K\) and, since \(\overline E_T<\infty\), equicontinuity.
The Arzel\`a--Ascoli theorem gives, along a subsequence,
\(x_{\mu_j}\to x\) uniformly on \([0,T]\), with \(x(0)=x_0\) and
\(x\in C([0,T];\mathbb R^n)\). Since \(K\) is closed, \(x(t)\in K\) for
all \(t\in[0,T]\). This also proves the relative compactness of the
family \(\{x_\mu\}\).

\smallskip
\textbf{Step~2}. \emph{Weak limits and the signed velocity identity.}
The bounds
\(\|w_{\mu_j}\|_{L^\infty},\|\eta_{\mu_j}\|_{L^\infty}\le M_K\)
and the weak-star sequential compactness (using the separability of
\(L^1(0,T;\mathbb R^n)\)) give, after a further extraction,
\[
w_{\mu_j}\rightharpoonup^\ast w,\qquad
\eta_{\mu_j}\rightharpoonup^\ast\eta
\quad\hbox{in }L^\infty(0,T;\mathbb R^n).
\]
Lemma~\ref{lem:proj-error} gives the residual bound \(\|\rho_{\mu_j}\|_{L^2(0,T)}\le\sqrt{E_{\nu_j}(T)}\to0\).
Since weak-star convergence in \(L^\infty\) implies weak convergence
in \(L^2\) on this finite interval, the identity
\(v_\mu=\rho_\mu-\eta_\mu\) gives
\(v_{\mu_j}\rightharpoonup-\eta\) in \(L^2(0,T;\mathbb R^n)\).
Moreover,
\[
\sup_{0\le t\le T}
\left\|\int_0^t\rho_{\mu_j}(s)\,ds\right\|
\le\sqrt{T}\,\|\rho_{\mu_j}\|_{L^2(0,T)}
\longrightarrow0.
\]
Passing to the limit in
\[
x_{\mu_j}(t)=x_{\nu_j,0}+\int_0^t
\bigl(w_{\mu_j}(s)-\eta_{\mu_j}(s)
      +\rho_{\mu_j}(s)\bigr)\,ds
\]
yields \(x(t)=x_0+\int_0^t(w(s)-\eta(s))\,ds\).
Since \(w-\eta\in L^\infty(0,T;\mathbb R^n)\), this establishes
\(x\in W^{1,\infty}(0,T;\mathbb R^n)\), in particular absolute
continuity on \([0,T]\), and \(\dot x=w-\eta\) almost everywhere.

\smallskip
\textbf{Step~3}. \emph{Identification of the limits.}

\noindent\emph{(a) Exact normal component.}
For \(t\in[t_k,t_{k+1})\), set \(\widehat x_\mu(t)=\widehat x_{k+1}\),
and set \(\widehat x_\mu(T)=x_{k_T}\).
Since \(\varepsilon_k\le h_k E_\nu(T)\), the projection-error
estimate and the rowwise continuity estimate give
\[
\begin{aligned}
\|\widehat x_\mu(t)-x_\mu(t)\|
&\le\sqrt{\varepsilon_k}
     +2M_Kh_k+\sqrt{E_\nu(T)h_k}\\
&\le2M_Kh_\nu+2\sqrt{h_\nu E_\nu(T)}.
\end{aligned}
\]
Hence \(\widehat x_{\mu_j}\to x\) uniformly on \([0,T]\).
For every \(z\in C\), \(\langle\eta_{\mu_j}(t),z-\widehat x_{\mu_j}(t)\rangle\le0\) for a.e.\ \(t\in[0,T]\).
We multiply by any nonnegative \(\varphi\in L^1(0,T)\) and integrate.
The uniform convergence of the base points and the boundedness of the
normals allow us to pass to the limit together with the weak-star
convergence, which gives \(\int_0^T\varphi(t)\langle\eta(t),z-x(t)\rangle\,dt\le0\).
Thus \(\langle\eta(t),z-x(t)\rangle\le0\) a.e. for each \(z\in C\).
We apply this first to a countable dense subset of \(C\), remove the union
of its null sets, and then use continuity in \(z\). It follows that
\(\eta(t)\in N_C(x(t))\) a.e. The residual has already vanished by Step~2.

\smallskip
\noindent\emph{(b) Selection at the left node.}
Define \(\underline x_\mu(t)=x_k\) for \(t\in[t_k,t_{k+1})\),
and \(\underline x_\mu(T)=x_{k_T}\). In particular, the last interval is
\([t_{k_T-1},T)\); its left node is used throughout, while the value at
\(T\) does not affect any a.e.\ assertion. Then
\[
w_\mu(t)\in F(\underline x_\mu(t))\quad\hbox{a.e.},\qquad
\|\underline x_\mu-x_\mu\|_{L^\infty}
\le 2M_Kh_\nu+\sqrt{h_\nu E_\nu(T)}.
\]
Consequently, \(\underline x_{\mu_j}\to x\) uniformly.
To identify the limiting selection, note that the weak-star convergence
of \(w_{\mu_j}\) implies weak convergence in \(L^2\) on the finite
interval. By Mazur's lemma~\cite[Corollary~3.8]{Brezis2011}, convex
combinations of tails of this sequence converge strongly in \(L^2\) to
\(w\), and a subsequence converges a.e. We discard the union of the
exceptional null sets for this pointwise convergence and for the countably
many inclusions \(w_{\mu_j}(t)\in F(\underline x_{\mu_j}(t))\).
For each remaining \(t\) and every \(\epsilon>0\), the upper
semicontinuity at \(x(t)\) gives
\[
F(\underline x_{\mu_j}(t))
\subset F(x(t))+\epsilon\overline B
\quad\hbox{for all sufficiently large }j.
\]
The right-hand side is closed and convex, so the tail convex combinations
and their pointwise limit belong to it. Letting \(\epsilon\downarrow0\)
yields \(w(t)\in F(x(t))\) a.e.

\smallskip
\textbf{Step~4}. \emph{Conclusion.}
Steps~2 and~3 give \(\dot x=w-\eta\), \(w\in F(x)\) and
\(\eta\in N_C(x)\) a.e., hence \(\dot x(t)\in F(x(t))-N_C(x(t))\) for a.e.\ \(t\in[0,T]\).
Together with \(x(0)=x_0\), this proves the solution property.

It remains to verify the additional convergence statements.
Suppose first that \(\overline Q_T<\infty\).
Then \(\|\rho_{\mu_j}\|_{L^\infty}\le\sqrt{\overline Q_T}\).
For \(\psi\in L^1(0,T;\mathbb R^n)\), choose
\(\psi_m\in L^2(0,T;\mathbb R^n)\) with
\(\|\psi_m-\psi\|_{L^1}\to0\). We have
\[
\left|\int_0^T
\langle\rho_{\mu_j}(t),\psi(t)\rangle\,dt\right|
\le
\sqrt{\overline Q_T}\,\|\psi-\psi_m\|_{L^1}
+\|\rho_{\mu_j}\|_{L^2}\|\psi_m\|_{L^2}.
\]
Letting first \(j\to\infty\) and then \(m\to\infty\) proves
\(\rho_{\mu_j}\rightharpoonup^\ast0\) in \(L^\infty\).
Thus \(v_{\mu_j}=\rho_{\mu_j}-\eta_{\mu_j}
\rightharpoonup^\ast-\eta\) in \(L^\infty\).
Finally, if \(Q_\nu(T)\to0\), then
\(\|\rho_{\mu_j}\|_{L^\infty}\le\sqrt{Q_{\nu_j}(T)}\to0\),
which proves the strong \(L^\infty\) convergence.
\end{proof}

Theorem~\ref{thm:limit_solution} requires a common compact convex set
containing the nodes. The next two corollaries supply this set, first
with~\eqref{A2} and then without it.

\begin{corollary}[With linear growth]
\label{cor:limit-linear-growth}
Assume~\eqref{A1}--\eqref{A2} and~\eqref{ass:G}. Fix \(T>0\) and consider
any finite-horizon mesh family from Definition~\ref{def:scheme_energy}
with a fixed initial state \(x_0\in C\), \(h_\nu\to0\), and
\(E_\nu(T)\to0\). Then the conclusions of Theorem~\ref{thm:limit_solution}
hold. Moreover, put \(p:=\proj_C(0)\) and
\(\overline E_T:=\sup_\nu E_\nu(T)\). Then every node lies in
\(K:=C\cap\overline B(p;R)\), where
\[
R:=e^{bT}\Bigl(\|x_0-p\|+(a+b\|p\|)T+\sqrt{T\,\overline E_T}\Bigr).
\]
\end{corollary}

\begin{proof}
For one row, we suppress \(\nu\) and write \(s_k:=\|x_k-p\|\).
The shadow-distance estimate~\eqref{eq:shadow-residual}, the
nonexpansiveness of \(\proj_C\), and~\eqref{A2} give
\[
s_{k+1}
\le\|x_k+h_k w_k-p\|+\sqrt{\varepsilon_k}
\le(1+bh_k)s_k
+(a+b\|p\|)h_k+\sqrt{\varepsilon_k}.
\]
Moreover, the Cauchy--Schwarz inequality yields
\[
\sum_{k<N_\nu}\sqrt{\varepsilon_k}
=\sum_{k<N_\nu}
\sqrt{h_k}\sqrt{\frac{\varepsilon_k}{h_k}}
\le\sqrt{T E_\nu(T)}
\le\sqrt{T\overline E_T}.
\]
Iterating the preceding recurrence (a discrete Gr\"onwall argument) and
using \(\prod_{k<N_\nu}(1+bh_k)\le e^{bT}\), we obtain
\[
\max_{0\le k\le N_\nu}s_k
\le e^{bT}\left(
\|x_0-p\|+(a+b\|p\|)T
+\sqrt{T\overline E_T}
\right)=R.
\]
Thus every node belongs to the common compact convex set
\(K=C\cap\overline B(p;R)\), and Theorem~\ref{thm:limit_solution} applies.
\end{proof}

When~\eqref{A2} does not hold, we replace it by the dissipativity
condition~\eqref{A3} at a center in \(C\).
Theorem~\ref{thm:confinement-anchor} then supplies the common compact
convex set \(K\), and we obtain the following convergence result.

\begin{corollary}[Without linear growth: a feasible center]
\label{cor:limit-anchored}
Assume the canonical interior-limit representation
\eqref{eq:A0}--\eqref{eq:decomp}, together with~\eqref{A1},
\eqref{ass:G}, and~\eqref{A3} at some \(p\in C\). Fix its rate
\(\gamma>0\) and a global dissipativity constant \(M_p\ge0\)
in~\eqref{eq:A3-anchor}. Fix \(T>0\) and consider a finite-horizon mesh
family from Definition~\ref{def:scheme_energy} such that
\[
x_{\nu,0}\to x_0\in C,\qquad h_\nu\to0,\qquad
E_\nu(T)\to0.
\]
Let \(R\), \(q_*\), and \(h_*\) be given by~\eqref{eq:anchor-thresholds}
for this family of initial states; note that \(r_{0,p}<\infty\) since
\(x_{\nu,0}\to x_0\). Assume, in addition, that \(Q_\nu(T)\le q_*^2\) for
all sufficiently large \(\nu\). Then every sequence of interpolants indexed
by \((\nu_j)\) with \(h_{\nu_j}\to0\) has a subsequence converging
uniformly on \([0,T]\) to a solution of~\eqref{eq:F-formulation} with
initial state \(x_0\). In particular, the inclusion admits a solution
from \(x_0\) on the finite horizon \([0,T]\).
\end{corollary}

\begin{proof}
We first apply the confinement theorem to the tail of the family, and then
obtain existence by exact projections. Let \(K:=C\cap\overline B(p;R)\).
Since \(h_\nu\to0\), we have \(h_\nu\le h_*\) for all sufficiently large
\(\nu\). Together with the assumed bound \(Q_\nu(T)\le q_*^2\), this
ensures that the thresholds of Theorem~\ref{thm:confinement-anchor} hold
for every \(\nu\ge\nu_0\), for some \(\nu_0\). Hence all nodes of the tail
lie in \(K\), and Theorem~\ref{thm:limit_solution} applies to this tail.
For the existence assertion, we take \(x_{\nu,0}=x_0\), exact projections,
and any refining family of finite partitions. The nonempty values of \(F\)
and the metric projection onto \(C\) make every row well defined.
\end{proof}

\subsection{Uniqueness and convergence of the whole family}
\label{subsec:whole-family}

This subsection is devoted to the uniqueness of solutions and to the
convergence of the whole family of interpolants. Theorem~\ref{thm:limit_solution}
gives subsequential convergence; uniqueness ensures that every subsequence
converges to the same limit. The following lemma provides the required
stability estimate under an all-pairs one-sided Lipschitz condition. In the
canonical setting, this condition holds whenever \(f\) is locally
Lipschitz on \(C\), since \(G\) is monotone, so the argument applies with
and without~\eqref{A2}.

\begin{lemma}[Local uniqueness and stability]
\label{lem:local-uniqueness-stability}
Let \(C\subset\mathbb R^n\) be nonempty, closed and convex, let \(K\subset C\)
be compact and convex, and suppose that \(F\) satisfies the all-pairs
one-sided Lipschitz condition
\begin{equation}\label{eq:OSL-FT}
\langle x-y,w-w'\rangle\le \ell_K\|x-y\|^2
\end{equation}
for all \(x,y\in K\), \(w\in F(x)\), and \(w'\in F(y)\), with some
\(\ell_K\in\mathbb R\). If \(x_1,x_2\) are absolutely continuous solutions
of~\eqref{eq:F-formulation} on \([0,T]\) with values in \(K\), then
\[
\|x_1(t)-x_2(t)\|\le
e^{\ell_Kt}\|x_1(0)-x_2(0)\|,
\qquad 0\le t\le T.
\]
In particular, a solution in \(K\) is unique for each initial state.
\end{lemma}

\begin{proof}
For almost every \(t\in[0,T]\), the solution relation supplies vectors
\(w_i(t)\) and \(n_i(t)\) such that \(w_i(t)\in F(x_i(t))\), \(n_i(t)\in N_C(x_i(t))\) and \(\dot x_i(t)=w_i(t)-n_i(t)\)
for \(i=1,2\). Setting \(z=x_1-x_2\), we obtain
\begin{align*}
\frac12\frac{d}{dt}\|z(t)\|^2
&=\langle z(t),w_1(t)-w_2(t)\rangle
-\langle z(t),n_1(t)-n_2(t)\rangle\\
&\le \ell_K\|z(t)\|^2,
\end{align*}
where the last inequality follows from the all-pairs condition and the
monotonicity of \(N_C\). Gr\"onwall's inequality proves the estimate.
\end{proof}

In the canonical setting, \(0\in N_C(x)\) and~\eqref{eq:decomp} give
\(G(x)\subset A(x)\) for every \(x\in C\). Hence \(G\) is monotone.
If \(f\) is Lipschitz on \(K\) with constant \(L_f\), then, for all
\(x,y\in K\), \(g_x\in G(x)\) and \(g_y\in G(y)\), \(\langle x-y,(f(x)-g_x)-(f(y)-g_y)\rangle\le L_f\|x-y\|^2\).
Thus \(F=f-G\) satisfies the all-pairs condition on \(K\) with
\(\ell_K=L_f\). If \(f\) is locally Lipschitz on \(C\), such a constant
exists on every compact convex set \(K\subset C\). Note that this
condition does not follow from~\eqref{ass:G} or from a dissipativity
condition alone. Combining Lemma~\ref{lem:local-uniqueness-stability} with
the compactness of the interpolants, we obtain the convergence of the whole
family under locally Lipschitz perturbations.

\begin{corollary}[Full-sequence convergence under Lipschitz perturbation]
\label{cor:full_convergence}
Assume the canonical interior-limit representation
\eqref{eq:A0}--\eqref{eq:decomp}, and suppose that \(f:C\to\mathbb R^n\)
is locally Lipschitz on \(C\). Fix \(T>0\), and assume the hypotheses of
Corollary~\ref{cor:limit-linear-growth} or of
Corollary~\ref{cor:limit-anchored}. Then~\eqref{eq:F-formulation} has
a unique solution \(x\in AC([0,T];\mathbb R^n)\) with \(x(0)=x_0\), and the
whole family of interpolants converges to it:
\[
x_\mu\to x
\quad\text{uniformly on }[0,T]\quad(\nu\to\infty).
\]
\end{corollary}

\begin{proof}
We split the proof into three steps.

\textbf{Step~1}. \emph{Existence and compactness.}
The applicable convergence corollary gives existence and the relative
compactness of a tail of the interpolants. Every subsequential limit
solves~\eqref{eq:F-formulation} with initial state \(x_0\).

\textbf{Step~2}. \emph{Uniqueness.}
Let \(z_1,z_2\) be two solutions on \([0,T]\) from \(x_0\). Their
continuous trajectories lie in a common compact convex set
\(K':=C\cap\overline B(0;r)\) for some \(r>0\). The local Lipschitz
continuity of \(f\) gives a Lipschitz constant on \(K'\), and together
with the monotonicity of \(G\), this yields the all-pairs one-sided
Lipschitz condition on \(K'\). Lemma~\ref{lem:local-uniqueness-stability}
therefore gives \(z_1=z_2\).

\textbf{Step~3}. \emph{Full convergence.}
Every subsequence of the interpolants admits a further subsequence
converging uniformly to a solution from \(x_0\), and uniqueness identifies
every such limit with \(x\). Consequently, the whole family converges
uniformly to \(x\) on \([0,T]\) as \(\nu\to\infty\).
\end{proof}

Corollary~\ref{cor:full_convergence} is a qualitative statement;
quantitative error estimates are established in
Section~\ref{sec:quantitative}.

\subsection{Global well-posedness}
\label{subsec:global-wellposed}

This subsection is devoted to the global well-posedness of
problem~\eqref{eq:F-formulation}. We treat the cases with and
without~\eqref{A2} separately, assuming throughout that \(f\) is locally
Lipschitz on \(C\).
{On every finite horizon, exact projections and refining mesh
families give a solution through Corollaries~\ref{cor:limit-linear-growth}
and~\ref{cor:limit-anchored}, and Lemma~\ref{lem:local-uniqueness-stability}
makes it unique. Hence, for a fixed initial state, the solutions on
overlapping intervals agree and define a unique global solution in
\(AC_{\mathrm{loc}}([0,\infty);\R^n)\). Under~\eqref{A2}, Gr\"onwall's
inequality confines the solutions to a ball on each finite horizon. The
dissipativity~\eqref{A3}, through Proposition~\ref{prop:cont_energy},
gives bounds independent of the horizon at either center, \(0\) or
\(p\in C\).}
For initial states in a fixed bounded subset of \(C\), these bounds give a
common compact convex confinement set, and
Lemma~\ref{lem:local-uniqueness-stability} then yields Lipschitz
continuous dependence on the initial data. The following theorem records
the global existence, uniqueness, confinement and stability conclusions
in both cases.

\begin{theorem}[Global well-posedness]
\label{thm:global-wellposed}
Assume the canonical interior-limit representation
\eqref{eq:A0}--\eqref{eq:decomp}, and suppose that \(f:C\to\mathbb R^n\)
is locally Lipschitz on \(C\).
{%
\par\smallskip\noindent\textnormal{(i)} \emph{With linear growth.}
Assume~\eqref{A1} and~\eqref{A2}. Then, for every $x_0\in C$, the
differential inclusion
\[
\dot x(t)\in F(x(t))-N_C(x(t)),
\qquad x(0)=x_0,
\]
has a unique solution $x\in AC_{\mathrm{loc}}([0,\infty);\mathbb R^n)$
with values in $C$. Let $x_1,x_2$ be the solutions from
$x_0^1,x_0^2\in C$, and fix $T>0$. Put $p:=\proj_C(0)$,
\[
R_T:=e^{bT}\Bigl(\max_{i=1,2}\|x_0^i-p\|+(a+b\|p\|)T\Bigr),
\qquad
\Omega_T:=C\cap\overline B(p;R_T).
\]
Then, for every Lipschitz constant $L_T$ of $f$ on $\Omega_T$,
\[
\|x_1(t)-x_2(t)\|\le e^{L_Tt}\|x_0^1-x_0^2\|,
\qquad 0\le t\le T.
\]
Moreover, if~\eqref{A3} holds at \rev{the center $0$}, with constants $\gamma>0$ and
$\widetilde M\ge0$ satisfying~\eqref{eq:A3_global} at \rev{this center}, then every 

solution satisfies
\[
\|x(t)\|\le\max\!\left\{\|x_0\|,\sqrt{\widetilde M/\gamma}\right\},
\qquad t\ge0,
\]
and, with
\[
R_\infty:=\max_{i=1,2}
\max\!\left\{\|x_0^i\|,\sqrt{\widetilde M/\gamma}\right\},
\qquad
\Omega_\infty:=C\cap\overline B(0;R_\infty),
\]
for every Lipschitz constant $L_\infty$ of $f$ on $\Omega_\infty$,
\[
\|x_1(t)-x_2(t)\|\le e^{L_\infty t}\|x_0^1-x_0^2\|,
\qquad t\ge0.
\]
}%
\par\smallskip\noindent\textnormal{(ii)} \emph{Without linear growth.}
Assume~\eqref{A1}, \eqref{ass:G}, and~\eqref{A3} at some \(p\in C\), and
fix its rate \(\gamma>0\) and a global dissipativity constant
\(M_p\ge0\) in~\eqref{eq:A3-anchor}. Then, for every \(x_0\in C\), the same
differential inclusion has a unique solution
\(x\in AC_{\mathrm{loc}}([0,\infty);\mathbb R^n)\) with values in \(C\), and
\[
\|x(t)-p\|\le R_{\infty,p}:=
\max\left\{\|x_0-p\|,\sqrt{M_p/\gamma}\right\},
\qquad t\ge0.
\]
Moreover, let \(x_1,x_2\) be the solutions from \(x_0^1,x_0^2\in C\), and
put
\[
R_{\infty,p}^i:=
\max\left\{\|x_0^i-p\|,\sqrt{M_p/\gamma}\right\},
\qquad
K:=C\cap\overline B\left(p;\max_{i=1,2}R_{\infty,p}^i\right).
\]
Then, for every Lipschitz constant \(L\) of \(f\) on \(K\),
\[
\|x_1(t)-x_2(t)\|
\le e^{Lt}\|x_0^1-x_0^2\|,
\qquad t\ge0.
\]
\end{theorem}

\begin{proof}
{%
We first prove~\textnormal{(i)}. Fix $x_0\in C$.
In the canonical setting, \eqref{A1}--\eqref{A2}
imply~\eqref{ass:G}; see Section~\ref{sec:setting}.
For every $T>0$, Corollary~\ref{cor:limit-linear-growth},
applied to a refining mesh family initialized at $x_0$
and using exact projections,
gives a solution $x^T(\cdot)\in AC([0,T];\R^n)$
of~\eqref{eq:F-formulation} on $[0,T]$ with $x^T(0)=x_0$.

Let $p:=\proj_C(0)$, and let $z$ be any solution
of~\eqref{eq:F-formulation} on $[0,T]$ with $z(0)=z_0\in C$.
By~\eqref{eq:proj-dynamics}, the nonexpansiveness of the
projection onto $T_C(z(t))$, and~\eqref{A2}, we have
$\|\dot z(t)\|\le a+b\|p\|+b\|z(t)-p\|$ for almost every $t$.
Gr\"onwall's inequality yields
\begin{equation}\label{eq:gw-linear}
\|z(t)-p\|\le e^{bT}\bigl(\|z_0-p\|+(a+b\|p\|)T\bigr),
\qquad 0\le t\le T.
\end{equation}
Hence every solution on $[0,T]$ from $x_0$ lies in the compact
convex set $C\cap\overline B\bigl(p;e^{bT}(\|x_0-p\|+(a+b\|p\|)T)\bigr)$.
Since $f$ is Lipschitz on this set,
Lemma~\ref{lem:local-uniqueness-stability} gives uniqueness
on $[0,T]$.

Let $0<T_1<T_2$. Then both $x^{T_1}(\cdot)$ and
$x^{T_2}(\cdot)|_{[0,T_1]}$ solve~\eqref{eq:F-formulation}
on $[0,T_1]$ with the same initial data $x_0$.
By uniqueness on finite horizons, they coincide on $[0,T_1]$.
Therefore the family $\{x^T(\cdot)\}_{T>0}$ is compatible
and defines a unique function $x(\cdot):[0,\infty)\to\R^n$
by setting $x(t):=x^T(t)$ for any $T>t$.
Then $x(\cdot)\in AC_{\mathrm{loc}}([0,\infty);\R^n)$,
$x(t)\in C$ for all $t\ge0$, and $x$ solves
\eqref{eq:F-formulation} on every finite interval,
hence on $[0,\infty)$.

Now let $x_0^1,x_0^2\in C$, let $x_1(\cdot),x_2(\cdot)$
be the corresponding solutions, and fix $T>0$.
By~\eqref{eq:gw-linear}, both trajectories remain in
$\Omega_T$ on $[0,T]$, with $R_T$ and $\Omega_T$ as in the
statement. Lemma~\ref{lem:local-uniqueness-stability},
applied with a Lipschitz constant $L_T$ of $f$ on $\Omega_T$,
gives the finite-horizon stability estimate.

Assume, in addition, that~\eqref{A3} holds at \rev{the center $0$}, with
constants $\gamma$ and $\widetilde M$ satisfying~\eqref{eq:A3_global}.
By Proposition~\ref{prop:cont_energy} with \rev{center $0$} and
$M_0=\widetilde M$, every solution satisfies
\[
\|x(t)\|^2
\le e^{-2\gamma t}\|x_0\|^2
+\frac{\widetilde M}{\gamma}\bigl(1-e^{-2\gamma t}\bigr)
\le\max\!\left\{\|x_0\|^2,\widetilde M/\gamma\right\},
\qquad t\ge0.
\]
In particular, $x_1$ and $x_2$ remain in $\Omega_\infty$
for all $t\ge0$. Lemma~\ref{lem:local-uniqueness-stability},
applied on every finite interval with a Lipschitz constant
$L_\infty$ of $f$ on $\Omega_\infty$, gives
\[
\|x_1(t)-x_2(t)\|
\le e^{L_\infty t}\|x_0^1-x_0^2\|,
\qquad t\ge0.
\]

}%
We now prove~\textnormal{(ii)}. Fix \(x_0\in C\).
For each \(T>0\), Corollary~\ref{cor:limit-anchored} gives a solution
\(x^T\) on \([0,T]\) with \(x^T(0)=x_0\); existence here comes from the
discrete approximation in that corollary.
Proposition~\ref{prop:cont_energy} with center \(p\) gives
\[
\|x^T(t)-p\|^2
\le e^{-2\gamma t}\|x_0-p\|^2
+\frac{M_p}{\gamma}(1-e^{-2\gamma t})
\le R_{\infty,p}^2,
\qquad 0\le t\le T.
\]
Thus every finite-horizon solution from \(x_0\) lies in the compact convex
set \(K_\infty=C\cap\overline B(p;R_{\infty,p})\), independently of \(T\).
The local Lipschitz continuity of \(f\) and
Lemma~\ref{lem:local-uniqueness-stability} give uniqueness on
\(K_\infty\). Consequently, as in~\textnormal{(i)}, solutions on
\([0,T_1]\) and \([0,T_2]\) agree on their common interval, and their
common extension is a solution
\(x\in AC_{\mathrm{loc}}([0,\infty);\R^n)\) satisfying the asserted bound.
Every global competitor from the same initial state lies in \(K_\infty\)
on each finite interval by Proposition~\ref{prop:cont_energy} with center
\(p\), so Lemma~\ref{lem:local-uniqueness-stability} makes the global
solution unique.

For two initial states, Proposition~\ref{prop:cont_energy} with center
\(p\) places the solutions in the common set
\(K=C\cap\overline B(p;\max_{i=1,2}R_{\infty,p}^i)\). Applying
Lemma~\ref{lem:local-uniqueness-stability} on every finite interval, with
a Lipschitz constant \(L\) of \(f\) on \(K\), proves the
continuous-dependence estimate and completes the proof.
\end{proof}

We next estimate the trajectory error in terms of the mesh size, the
accumulated projection-error budget, and the discrepancy of the initial
states.

\section{Quantitative stability and error estimates}
\label{sec:quantitative}

This section is devoted to quantitative stability and finite-horizon error
estimates for the catching-up scheme. Sections~\ref{sec:energy}
and~\ref{sec:convergence} provide confinement and convergence, but no rate.
Here we derive an explicit trajectory-error bound of order
\(\sqrt{h_\nu+E_\nu(T)}\) by a differential energy comparison.

Theorem~\ref{thm:limit_solution} uses no Lipschitz-type condition on \(F\),
and it gives neither uniqueness nor a rate. To obtain a rate, we add one
assumption, namely the all-pairs one-sided Lipschitz
condition~\eqref{eq:OSL-FT} on the set \(K\) containing the nodes. Recall
from Section~\ref{subsec:whole-family} that, in the canonical setting, this
condition holds whenever \(f\) is locally Lipschitz on \(C\), and that it
yields uniqueness and the convergence of the whole family.
Theorem~\ref{thm:trajectory-rate} below makes this convergence
quantitative; its proof combines~\eqref{eq:OSL-FT} with the one-step
estimates of Lemma~\ref{lem:proj-error}. The theorem assumes that
solutions exist from each initial point and take values in \(K\); under the
confinement results of Section~\ref{sec:energy}, such solutions are
obtained by applying Theorem~\ref{thm:limit_solution} to a refining family
from each initial point. Note also that, for a fixed \(K\), the rate
depends on the budget \(E_\nu(T)\) but not on \(Q_\nu(T)\); the budget
\(Q_\nu(T)\) is used only in the confinement step that produces \(K\).

Fix \(T>0\). Throughout this section, \(K\) denotes a compact convex
subset of \(C\) containing every node and every continuous trajectory
under comparison. Under~\eqref{A2}, we take
\(K=C\cap\overline B(0;\rho_T)\) with \(\rho_T\ge R_T\), where \(R_T\) is
the radius of Corollary~\ref{cor:trajectory-rate-linear}; this radius
bounds every node and every comparison solution, and it uses only
\(\sup_\nu E_\nu(T)<\infty\). Under~\eqref{A3} at a center \(p\in C\), we
take \(K=C\cap\overline B(p;R)\) from Theorem~\ref{thm:confinement-anchor}
for rows satisfying \(h_\nu\le h_*\) and \(Q_\nu(T)\le q_*^2\), and
Proposition~\ref{prop:cont_energy} with center \(p\) gives the
corresponding confinement of the continuous trajectories. In both cases,
the radii are chosen using a common bound on all initial states under
comparison.

Under~\eqref{A1} and~\eqref{ass:G}, we retain the finite local bound
\(M_K\) of Lemma~\ref{lem:bounded_equicontinuous}, and we assume the
all-pairs condition~\eqref{eq:OSL-FT} on \(K\), as in
Lemma~\ref{lem:local-uniqueness-stability}. The accumulated projection
error is measured by \(E_\nu(T)\). The comparison also uses a local bound
on the velocity of each solution: if \(x\in AC([0,T];\mathbb R^n)\) is a
solution of~\eqref{eq:F-formulation} with \(x(t)\in K\) on \([0,T]\), then \(\|\dot x(t)\|\le M_K\) for a.e.\ \(t\in[0,T]\).
Indeed, for almost every \(t\), we have \(\dot x(t)=w-n\) with
\(w\in F(x(t))\) and \(n\in N_C(x(t))\), and
\(\langle n,\dot x(t)\rangle=0\), as shown in Section~\ref{sec:setting}.
Consequently,
\(\|\dot x(t)\|^2=\langle w,\dot x(t)\rangle\le M_K\|\dot x(t)\|\).

For the dependence on the initial state, we use the stability estimate of
Lemma~\ref{lem:local-uniqueness-stability}: any two solutions
\(x_1,x_2:[0,T]\to K\) satisfy
\(\sup_{0\le t\le T}\|x_1(t)-x_2(t)\|\le e^{\rev{\max\{\ell_K,0\}}T}\|x_1(0)-x_2(0)\|\).
Combining it with a differential energy comparison, we arrive at the main
trajectory-error estimate of the paper. Fix \(T>0\) and a compact convex
set \(K\subset C\), let \(\ell_K\) be a constant for which \(F\) satisfies
the all-pairs condition~\eqref{eq:OSL-FT} on \(K\), and let
\(M_K:=\sup\{\|w\|:x\in K,\ w\in F(x)\}\) be the bound of
Lemma~\ref{lem:proj-error}. We use the constants
\begin{equation}\label{eq:rate-constants}
\ell_+:=\max\{\ell_K,0\},\quad
\alpha:=4\ell_++1,\quad
\beta:=M_K^2(9+8\ell_+T),\quad
\chi:=1+3T+8\ell_+T^2,
\end{equation}
which depend only on \(T\), \(M_K\), and \(\ell_K\).

\begin{theorem}[Finite-horizon trajectory error under a one-sided Lipschitz condition]
\label{thm:trajectory-rate}
Assume~\eqref{A1} and~\eqref{ass:G}, and let \(T\), \(K\), and \(\ell_K\) be
as above. Consider a finite-horizon mesh family from
Definition~\ref{def:scheme_energy} with initial points
\(x_{\nu,0}\in C\), and fix \(x_0\in C\). Suppose that every node of every
row lies in \(K\). Suppose also that~\eqref{eq:F-formulation} has a
solution on \([0,T]\) from each of the points \(x_{\nu,0}\) and \(x_0\), and
that every such solution takes values in \(K\). Then the solution \(x\)
from \(x_0\) and the solution from each \(x_{\nu,0}\) are unique, and every
row interpolant satisfies
\begin{equation}\label{eq:trajectory-rate}
\|x_\mu-x\|_{L^\infty(0,T)}
\le e^{\ell_+T}\|x_{\nu,0}-x_0\|
+e^{\alpha T/2}\sqrt{\beta T h_\nu+\chi E_\nu(T)}
\end{equation}
with the constants~\eqref{eq:rate-constants}. In particular, with
\(K_T:=e^{\alpha T/2}\sqrt{\max\{\beta T,\chi\}}\),
\[
\|x_\mu-x\|_{L^\infty(0,T)}
\le e^{\ell_+T}\|x_{\nu,0}-x_0\|
+K_T\sqrt{h_\nu+E_\nu(T)}.
\]
The constants are independent of the row and of \(Q_\nu(T)\).
Consequently, if \(x_{\nu,0}\to x_0\), \(h_\nu\to0\), and
\(E_\nu(T)\to0\), then \(x_\mu\to x\) uniformly on \([0,T]\) as
\(\nu\to\infty\).
\end{theorem}

\begin{proof}
Lemma~\ref{lem:local-uniqueness-stability} gives uniqueness for all the
solutions under consideration. Let \(x^\nu\) be the solution from
\(x_{\nu,0}\), and set \(e:=x_\mu-x^\nu\) and \(D:=\|e\|^2\). Then
\(D\in AC([0,T])\) and \(D(0)=0\). We split the proof into five steps:
Step~1 estimates the interpolant differences and sets up the energy
inequality, and Steps~2--5 bound the four terms on its right-hand side.

\textbf{Step~1}. \emph{Interpolant differences and energy setup.}
Fix a row and put \(r_k:=\|\rho_k\|\). By Lemma~\ref{lem:proj-error}, \(\|w_k\|\le M_K\), \(\|\eta_k\|\le M_K\), \(\|\widehat x_{k+1}-x_k\|\le M_Kh_k\),
and, for \(t\in[t_k,t_{k+1})\), \(\|x_\mu(t)-x_k\|,\ \|x_\mu(t)-\widehat x_{k+1}\|\le h_k(M_K+r_k)\).
For almost every \(t\in(t_k,t_{k+1})\), both derivatives exist and the
solution inclusion holds. We fix such a time and choose vectors \(w\in F(x^\nu(t))\) and \(n\in N_C(x^\nu(t))\) with \(\dot x^\nu(t)=w-n\).
Since \(x^\nu(t)\in K\), we have \(\|w\|\le M_K\). The velocity identity
following Definition~\ref{def:interpolants} gives
\[
\tfrac12\dot D
=-\langle\eta_k,e\rangle+\langle n,e\rangle
+\langle w_k-w,e\rangle+\langle\rho_k,e\rangle,
\]
and we estimate these four terms separately.

\textbf{Step~2}. \emph{Discrete shadow normal.}
Since \(x^\nu(t)\in C\) and \(\eta_k\in N_C(\widehat x_{k+1})\), \(-\langle\eta_k,\widehat x_{k+1}-x^\nu(t)\rangle\le0\).
Hence the shadow-gap estimate yields \(-\langle\eta_k,e\rangle\le\|\eta_k\|\,\|x_\mu(t)-\widehat x_{k+1}\|\le M_Kh_k(M_K+r_k)\).

\textbf{Step~3}. \emph{Continuous normal.}
Every affine segment joins two points of the convex set \(C\), so
\(x_\mu(t)\in C\). Therefore, \(\langle n,e\rangle=\langle n,x_\mu(t)-x^\nu(t)\rangle\le0\).

\textbf{Step~4}. \emph{Drift comparison at the left node.}
Put \(d_k(t):=x_k-x^\nu(t)\). The all-pairs condition at
\(x_k,x^\nu(t)\in K\) and \(\|w_k-w\|\le2M_K\) give
\begin{align*}
\langle w_k-w,e\rangle
&\le\ell_+\|d_k(t)\|^2
   +2M_K\|x_\mu(t)-x_k\|\\
&\le2\ell_+D
   +2\ell_+h_k^2(M_K+r_k)^2
   +2M_Kh_k(M_K+r_k).
\end{align*}
Here we used \(\ell_+\ge0\) and \(d_k(t)=e(t)-(x_\mu(t)-x_k)\).

\textbf{Step~5}. \emph{Projection residual.}
Young's inequality gives \(\langle\rho_k,e\rangle\le\tfrac12r_k^2+\tfrac12D\).
Combining the four estimates yields
\[
\dot D
\le\alpha D
+6M_Kh_k(M_K+r_k)
+4\ell_+h_k^2(M_K+r_k)^2+r_k^2.
\]
Moreover, \(6M_Kh_k r_k\le3M_K^2h_k+3h_k r_k^2\) and \((M_K+r_k)^2\le2M_K^2+2r_k^2\).
Since \(h_k\le T\), it follows that
\begin{align*}
\dot D
&\le\alpha D
 +(9M_K^2+8\ell_+M_K^2T)h_k
 +(1+3T+8\ell_+T^2)r_k^2\\
&=\alpha D+\beta h_k+\chi r_k^2
\qquad\text{a.e. on }[t_k,t_{k+1}).
\end{align*}
The shadow-distance estimate and the mesh definition give
\[
\int_0^T\|\rho_\mu(t)\|^2\,dt
=\sum_{k<N_\nu}h_k r_k^2
\le\sum_{k<N_\nu}
\frac{\varepsilon_{\nu,k}}{h_{\nu,k}}
=E_\nu(T),
\qquad
\sum_{k<N_\nu}h_k^2\le T h_\nu.
\]
Applying the integrating-factor inequality and using \(D(0)=0\), we obtain
\[
\sup_{0\le t\le T}D(t)
\le e^{\alpha T}
\left(
\beta\sum_{k<N_\nu}h_k^2
+\chi\int_0^T\|\rho_\mu(t)\|^2\,dt
\right)
\le e^{\alpha T}\bigl(\beta T h_\nu+\chi E_\nu(T)\bigr).
\]
Consequently,
\[
\|x_\mu-x^\nu\|_{L^\infty(0,T)}
\le e^{\alpha T/2}
\sqrt{\beta T h_\nu+\chi E_\nu(T)}.
\]
Lemma~\ref{lem:local-uniqueness-stability} also gives \(\|x^\nu-x\|_{L^\infty(0,T)}\le e^{\ell_+T}\|x_{\nu,0}-x_0\|\),
and the triangle inequality proves~\eqref{eq:trajectory-rate}.
The second estimate follows from \(\beta T h_\nu+\chi E_\nu(T)\le\max\{\beta T,\chi\}(h_\nu+E_\nu(T))\)
and the definition of \(K_T\). Finally, the right-hand side
of~\eqref{eq:trajectory-rate} tends to zero when \(x_{\nu,0}\to x_0\),
\(h_\nu\to0\), and \(E_\nu(T)\to0\), which proves the uniform convergence
on \([0,T]\).
\end{proof}

To apply Theorem~\ref{thm:trajectory-rate}, we need a common compact convex
set for the nodes and the comparison solutions. The next corollary provides
it, with~\eqref{A2} and without it.

\begin{corollary}[Trajectory-error estimates with and without linear growth]
\label{cor:trajectory-rate-linear}\label{cor:trajectory-rate-anchored}
Fix \(T>0\) and initial points \(x_{\nu,0}\in C\) with
\(x_{\nu,0}\to x_0\in C\), and consider a finite-horizon mesh family from
Definition~\ref{def:scheme_energy}, initialized at these points, with
\(h_\nu\to0\) and \(E_\nu(T)\to0\). Assume one of the following.
\par\smallskip\noindent\textnormal{(i)} \emph{With linear growth.}
Assumptions~\eqref{A1}, \eqref{A2}, and~\eqref{ass:G} hold. Put
\(\rev{r_0}:=\sup_\nu\|x_{\nu,0}\|<\infty\),
\(\overline E_T:=\sup_\nu E_\nu(T)<\infty\), \(p:=\proj_C(0)\), and
\[
R_T:=\|p\|+e^{bT}
\left(\rev{r_0}+(a+b\|p\|)T+\sqrt{T\overline E_T}\right),
\qquad M_T:=a+bR_T.
\]
Suppose that~\eqref{eq:OSL-FT} holds with constant \(\ell_T\) on
\(\Omega_T=C\cap\overline B(0;\rho_T)\) for some \(\rho_T\ge R_T\), and let
\(\ell_+,\alpha,\beta,\chi\) be given by~\eqref{eq:rate-constants} with
\(\ell_T\) in place of \(\ell_K\) and \(M_T\) in place of \(M_K\). These
constants depend only on \(T,a,b,\rev{r_0},d_C(0),\overline E_T\), and \(\ell_T\).
\par\smallskip\noindent\textnormal{(ii)} \emph{Without linear growth: a
feasible center.} The canonical representation~\eqref{eq:A0}--\eqref{eq:decomp},
\eqref{A1}, \eqref{ass:G}, and~\eqref{A3} at some \(p\in C\) hold, with rate
\(\gamma>0\) and a global dissipativity constant \(M_p\ge0\)
in~\eqref{eq:A3-anchor}, and \(f\) is locally Lipschitz on \(C\). Let
\(R\), \(h_*\), \(q_*\), and \(K=C\cap\overline B(p;R)\) be given by
Theorem~\ref{thm:confinement-anchor} for these initial points, and suppose
that \(h_\nu\le h_*\) and \(Q_\nu(T)\le q_*^2\); in particular,
\(\overline Q_T\le q_*^2<\infty\). Let \(\ell_K\) be a Lipschitz constant of
\(f\) on \(K\), and let \(\ell_+,\alpha,\beta,\chi\) be given
by~\eqref{eq:rate-constants}.
\par\smallskip\noindent
Then, in both cases, \eqref{eq:F-formulation} has a unique solution \(x\)
from \(x_0\) on \([0,T]\), and every row satisfies
\begin{align*}
\|x_\mu-x\|_{L^\infty(0,T)}
&\le e^{\ell_+T}\|x_{\nu,0}-x_0\|
 +e^{\alpha T/2}\sqrt{\beta T h_\nu+\chi E_\nu(T)}\\
&\le e^{\ell_+T}\|x_{\nu,0}-x_0\|
 +K_T\sqrt{h_\nu+E_\nu(T)},
\end{align*}
where \(K_T:=e^{\alpha T/2}\sqrt{\max\{\beta T,\chi\}}\). The constants are
independent of the row, and \(x_\mu\to x\) uniformly on \([0,T]\) as
\(\nu\to\infty\). If \(x_{\nu,0}=x_0\) for every \(\nu\), then the
initial-data term vanishes.
\end{corollary}

\begin{proof}
We first prove~\textnormal{(i)} in three steps.

\textbf{Step~1}. \emph{Confinement of nodes and solutions.}
Since \(p=\proj_C(0)\), the nonexpansiveness of the projection gives
\(\|z-p\|\le\|z\|\) for every \(z\in C\). Hence
\(\|x_{\nu,0}-p\|\le \rev{r_0}\) and \(\|x_0-p\|\le \rev{r_0}\).
Arguing as in the proof of Corollary~\ref{cor:limit-linear-growth}, with
the initial distance bounded by \(\rev{r_0}\), we obtain
\[
\max_{0\le k\le N_\nu}\|x_{\nu,k}-p\|
\le e^{bT}
\left(\rev{r_0}+(a+b\|p\|)T+\sqrt{T\overline E_T}\right).
\]
Consequently, \(\|x_{\nu,k}\|\le R_T\) for every node, and all nodes belong
to the compact convex set \(K:=C\cap\overline B(0;R_T)\).
For any \(C\)-valued solution \(z\) on \([0,T]\) from \(x_0\) or
\(x_{\nu,0}\), the Gr\"onwall bound~\eqref{eq:gw-linear}\rev{, whose derivation does not use~\eqref{A3},} gives
\(\|z(t)-p\|\le e^{bT}(\rev{r_0}+(a+b\|p\|)T)\) for \(0\le t\le T\). Hence every
comparison solution on \([0,T]\) has values in \(K\).

\textbf{Step~2}. \emph{Existence of comparison solutions.}
For each initial point \(x_0\) or \(x_{\nu,0}\), we construct an auxiliary
refining family with exact projections. The nonempty values of \(F\) and
the metric projection onto \(C\) make every row well defined. The discrete
bound of Step~1 places all auxiliary nodes in \(K\), and their accumulated
budgets are zero. Thus Theorem~\ref{thm:limit_solution} supplies a solution
on \([0,T]\) from each initial point.

\textbf{Step~3}. \emph{Application of the trajectory estimate.}
Since \(K\subset\Omega_T\), the all-pairs condition holds on \(K\) with
constant \(\ell_T\), and \(M_K\le M_T\) by~\eqref{A2}. Hence the constants
of Theorem~\ref{thm:trajectory-rate} on \(K\) are bounded by \(\alpha\),
\(\beta\), and \(\chi\) above, which gives both estimates; the uniform
convergence follows from the assumed limits.

We now prove~\textnormal{(ii)}. Theorem~\ref{thm:confinement-anchor} places
every node in \(K\), and Theorem~\ref{thm:global-wellposed}\textnormal{(ii)}
gives existence from each initial state. Since \(x_{\nu,0}\to x_0\), the
bound on the initial distances to \(p\) used to define \(R\) also covers
\(x_0\), so Proposition~\ref{prop:cont_energy} with center \(p\) and the
definition of \(R\) place every comparison solution in \(K\). Since \(G\) is
monotone and \(f\) is Lipschitz on \(K\), the all-pairs condition holds
there with constant \(\ell_K\). Theorem~\ref{thm:trajectory-rate} gives the
stated estimates, and the assumed limits yield the uniform convergence.
\end{proof}

We next apply the trajectory-error estimates to an elastic network with
unilateral stops.

\section{Application: an elastic network with unilateral stops}
\label{sec:noncompact-springs}

This section is devoted to an application of our results to the relaxation
of an elastic network with unilateral stops. Its constraint set is
noncompact, and the field has no linear growth. We verify~\eqref{A3} at a
feasible center, apply the main theorems, and report two numerical
experiments.

\paragraph{Model origin and assumptions}
We consider the viscous relaxation of guided elastic elements coupled by
hardening springs and subject to rigid lower displacement stops. The
constitutive ingredients are established: \rev{a linear-plus-cubic spring law is
derived in~\cite{BoisseauSpring2013} for an inertial vibration-harvesting
model}, while normal-cone
descriptions of unilateral constraints are
classical~\cite{MoreauContact2011,Marques1993,LeineVanDeWouw2008}. We
combine these ingredients in a finite-dimensional first-order test
problem, for which we choose the network, the constant loading and the
numerical parameters.

Fix a connected graph with \(n\ge2\) vertices and oriented incidence
matrix \(B\in\R^{m\times n}\). The scalar \(x_i\) denotes the displacement
of element \(i\) along its guide, measured from its lower stop. We choose a
compatible undeformed reference at \(x=0\), so that an oriented edge
\(e=(i,j)\) has generalized deformation \((Bx)_e=x_j-x_i\). The scalar
coupling describes relative displacements along the guides, rather than
changes in the Euclidean length of a spatial spring. The edge force law is \(r_e(s)=a_es+c_es^3\), with \(a_e,c_e>0\),
and a grounding spring contributes \(-\kappa x_i\), with \(\kappa>0\).

We neglect inertia and assume identical positive scalar viscous
coefficients, normalized to one by rescaling time, and we adopt the
overdamped equations directly as a viscous relaxation model. For a constant
applied load \(b\in\R^n\) and a lower-stop reaction \(R\), the force balance
is
\[
\dot x=b-\kappa x-B^\top r(Bx)+R,
\qquad
x_i\ge0,\quad R_i\ge0,\quad x_iR_i=0.
\]
Thus \(-R\in N_C(x)\) for \(C=\R_+^n\). Only lower stops are imposed: the
bounded regions obtained below follow from the dynamics and the numerical
conditions, rather than from additional upper constraints.

The polynomial force law serves as an idealized constitutive model.
Grounding controls the common translation mode, since
\(B\boldsymbol1=0\), and supplies the strict restoring term used in the
dissipativity at a feasible center. The cubic component produces
superlinear growth. The loads chosen below also require a positive global
dissipativity constant at the reference \(p=0\). All arguments are finite
dimensional, with the graph fixed.

\paragraph{Model and canonical decomposition}
We define the constrained elastic energy by
\[
\Psi(x)=
\begin{cases}
\displaystyle
\frac{\kappa}{2}\|x\|^2+
\sum_{e=1}^{m}
\left[
\frac{a_e}{2}(Bx)_e^2+
\frac{c_e}{4}(Bx)_e^4
\right],
&x\in C,\\[2mm]
+\infty,&x\notin C.
\end{cases}
\]
For \(x_0\in C\), the force balance becomes
\[
\dot x(t)+A(x(t))\ni b,
\qquad x(0)=x_0,
\qquad A:=\partial\Psi.
\]
Writing \(\Phi\) for the polynomial energy, we have \(\Psi=\Phi+I_C\).
Since \(\Phi\) is finite, continuously differentiable and convex, the
subdifferential sum rule yields~\cite{Bauschke}
\[
A=G+N_C,
\qquad
G(x)=\{\nabla\Phi(x)\},
\qquad
\nabla\Phi(x)=\kappa x+B^\top r(Bx).
\]
The operator \(A\) is maximally monotone, and \(\operatorname{dom}A=\overline{\operatorname{dom}A}=C=\R_+^n\),
which is noncompact. On \(\operatorname{int}C\), the operator \(A\) has the
unique value \(\nabla\Phi(x)\). The continuity of \(\nabla\Phi\) and the
density of \(\operatorname{int}C\) in \(C\) therefore identify \(G\) with
the canonical interior-limit component, including on \(\partial C\). The
monotone component \(G\) represents the elastic forces, whereas \(N_C\)
represents the constraint reaction. Figure~\ref{fig:noncompact-mechanics}
shows the network.

\begin{figure}[H]
\centering
\includegraphics[width=.8\linewidth]{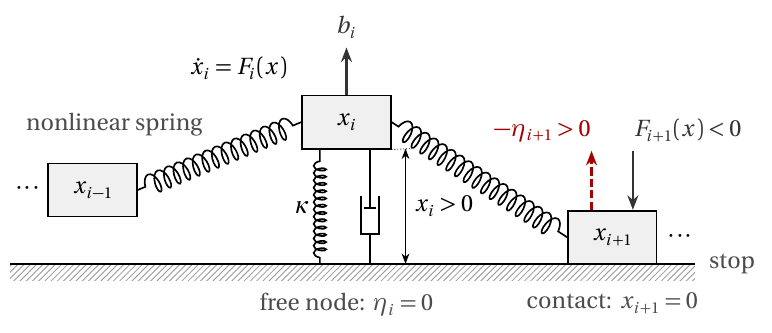}
\caption{Overdamped elastic network with grounding springs, constant loads,
lower stops, and no upper bound.}
\label{fig:noncompact-mechanics}
\end{figure}

\paragraph{Algorithms}
We use two algorithms in the experiments. Algorithm~\ref{alg:spring-split}
is the split scheme studied in this paper; it specializes
Algorithm~\ref{alg:catching-up} to the elastic network and uses an
auxiliary point \(q\in C\) to construct feasible approximate corrections.
In the experiments, \(q=\boldsymbol1/2\). Algorithm~\ref{alg:spring-resolvent}
is the full-resolvent discretization, which we use for comparison.

\begin{algorithm}[H]
\caption{Split catching-up with feasible approximate projections}
\label{alg:spring-split}
\begin{algorithmic}[1]
\Require Initial state $x_0\in C=\R_+^n$;
grid $0=t_0<\cdots<t_N=T$;
tolerances $\varepsilon_k\ge0$;
auxiliary point $q\in C$.
\Ensure Feasible nodes $(x_k)_{k=0}^{N}$.
\For{$k=0,\ldots,N-1$}
    \State $h_k\gets t_{k+1}-t_k$
    \State $y\gets x_k+h_k(b-\nabla\Phi(x_k))$
    \State $\widehat x\gets\max\{y,0\}$
    \Comment{Componentwise projection}
    \State $d\gets q-\widehat x$
    \If{$\varepsilon_k=0$ or $\|d\|=0$}
        \State $x_{k+1}\gets\widehat x$
    \Else
        \State $a\gets\langle\widehat x-y,d\rangle$
        \State $\displaystyle
        \theta\gets\min\left\{1,
        \frac{\varepsilon_k}
        {a+\sqrt{a^2+\|d\|^2\varepsilon_k}}\right\}$
        \State $x_{k+1}\gets\widehat x+\theta d$
    \EndIf
\EndFor
\State \Return $(x_k)_{k=0}^{N}$
\end{algorithmic}
\end{algorithm}

The projection characterization gives \(a\ge0\). Since
\(q,\widehat x\in C\) and \(0\le\theta\le1\), the correction is feasible.
In the nontrivial correction branch, \(\|x_{k+1}-y\|^2-d_C(y)^2=2a\theta+\|d\|^2\theta^2\le\varepsilon_k\),
and the other branch is the exact projection. Thus
\(x_{k+1}\in\proj_C^{\varepsilon_k}(y)\). The exact shadow provides a
reference for generating controlled projection errors in these
experiments, so this construction tests the sensitivity to errors rather
than the cost of the projection.

For comparison, we use the full-resolvent discretization
\(x_{k+1}^{\rm res}=J_{h_kA}(\xi_k)\), with \(\xi_k=x_k^{\rm res}+h_kb\), where
the superscript \(\mathrm{res}\) labels the resolvent iterates. Since
\(h_kN_C=N_C\), the definition of the resolvent shows that
\(z=J_{h_kA}(\xi_k)\) if and only if
\(0\in z-\xi_k+h_k\nabla\Phi(z)+N_C(z)\), that is, \(z\) minimizes over \(C\)
the function
\[
J_k(v):=\frac12\|v-\xi_k\|^2+h_k\Phi(v),
\]
whose gradient is \(g_k(z)=z-\xi_k+h_k[\kappa z+B^\top r(Bz)]\). The
objective is coercive and \((1+h_k\kappa)\)-strongly convex, so the
minimizer is unique; this is the classical proximal
interpretation~\cite{ParikhBoyd2014,Bauschke}. This recursion is exact, and
Algorithm~\ref{alg:spring-resolvent} realizes it up to the solver tolerance
\(\tau\). At each step, the full-resolvent implementation solves a coupled
nonlinear complementarity problem, whereas a split step requires one
explicit evaluation of the elastic force and a componentwise projection.

\begin{algorithm}[H]
\caption{Numerical full-resolvent discretization}
\label{alg:spring-resolvent}
\begin{algorithmic}[1]
\Require Initial state $x_0^{\rm res}\in C$;
grid $0=t_0<\cdots<t_N=T$;
solver tolerance $\tau>0$.
\Ensure Computed resolvent nodes $(x_k^{\rm res})_{k=0}^{N}$,
or a solver-failure report.
\For{$k=0,\ldots,N-1$}
    \State $h_k\gets t_{k+1}-t_k$
    \State $\xi_k\gets x_k^{\rm res}+h_kb$
    \State Initialize $z\gets x_k^{\rm res}$
    \State Apply the projected Newton solver to
    $\displaystyle\min_{z\in C}
    J_k(z)=\frac12\|z-\xi_k\|^2+h_k\Phi(z)$
    \If{$\|z-\proj_C(z-g_k(z))\|_\infty>\tau$}
        \State \Return Solver failure
    \EndIf
    \State $x_{k+1}^{\rm res}\gets z$
\EndFor
\State \Return $(x_k^{\rm res})_{k=0}^{N}$
\end{algorithmic}
\end{algorithm}

The projected Newton solver, with the Hessian of \(J_k\) and the optimality
conditions it uses, is described in the supplementary material. A time step
is accepted only when the residual criterion is satisfied, and its
tolerance \(\tau\) is distinct from the tolerance \(\varepsilon_k\) of the
feasible approximate projection.

\paragraph{Verification of (A3) at a feasible center}
The constant perturbation \(f=b\) satisfies~\eqref{A1}, and the continuous
singleton-valued map \(G\) satisfies~\eqref{ass:G}. Moreover,
\[
\nabla^2\Phi(x)
=\kappa I+
B^\top\operatorname{diag}
\bigl(a_e+3c_e(Bx)_e^2\bigr)B
\succeq\kappa I.
\]
Consequently, identifying singleton-valued maps with their values, the
effective field \(F=b-\nabla\Phi\) satisfies \(\langle x-y,F(x)-F(y)\rangle\le-\kappa\|x-y\|^2\) for all \(x,y\in C\).
Thus~\eqref{eq:OSL-FT} holds with \(\ell_K=-\kappa\) on every compact
convex subset of \(C\). For a coordinate ray \(te_i\in C\) incident to an
edge,
\[
\|F(te_i)\|
\ge t^3\sum_{e\sim i}c_e-\|b\|,
\qquad t\ge0,
\]
so the linear growth condition~\eqref{A2} does not hold. Nevertheless, the
field satisfies the dissipativity estimate
\begin{align*}
\langle x,F(x)\rangle
&=\langle b,x\rangle-\kappa\|x\|^2
-\sum_e\bigl(a_e(Bx)_e^2+c_e(Bx)_e^4\bigr)\\
&\le\frac{\|b\|^2}{2\kappa}
-\frac{\kappa}{2}\|x\|^2.
\end{align*}
For the orthant, the tangential projection changes only the coordinates
for which \(x_i=0\). Hence \(\langle x,\proj_{T_C(x)}F(x)\rangle=\langle x,F(x)\rangle\),
and~\eqref{A3} and its global form~\eqref{eq:A3-anchor} hold with \(p=0\), \(\gamma=\kappa/2\), and \(M_p=\|b\|^2/(2\kappa)\).
Theorem~\ref{thm:global-wellposed}\textnormal{(ii)} gives a unique global
solution, and the continuous energy estimate at \(p=0\) yields
\[
\|x(t)\|^2
\le e^{-\kappa t}\|x_0\|^2
+\frac{\|b\|^2}{\kappa^2}(1-e^{-\kappa t}).
\]
Lemma~\ref{lem:local-uniqueness-stability} also gives \(\|x_1(t)-x_2(t)\|\le e^{-\kappa t}\|x_1(0)-x_2(0)\|\).
Note that feasibility alone cannot supply this norm bound on \(\R_+^n\).

\paragraph{Numerical data and proved localization}
We take a chain with \(n=12\), \(T=3\), \(\kappa=1\), \(a_e=c_e=s\in\{1,10,100\}\), \(x_0=\boldsymbol1\), and \(b_i=-2+3(i-1)/11\) for \(i=1,\ldots,12\).
Then \(\|b\|^2=150/11\) and \(M_p=75/11>0\). Since \(b_{12}=1\), \(\langle te_{12},F(te_{12})\rangle=t+O(t^2)>0\)
for sufficiently small \(t>0\). Thus the global
inequality~\eqref{eq:A3-anchor} cannot hold with \(p=0\) and \(M_p=0\) for
these data.

Theorem~\ref{thm:confinement-anchor} supplies sufficient mesh and
tolerance thresholds. For the present experiments, a sharper argument
gives \(K=[0,1]^{12}\)
without introducing an upper constraint. Indeed, for a continuous
solution, we write \(\dot x=F(x)-n\) with \(n\in N_C(x)\) and set
\(v=(x-\boldsymbol1)^+\) componentwise. Since \(b_i\le1\),
\(\langle v,n\rangle=0\), and \(\langle v,B^\top r(Bx)\rangle=\sum_e(Bv)_e r_e((Bx)_e)\ge0\),
we obtain \(\tfrac12\tfrac{d}{dt}\|v\|^2\le-\|v\|^2\) a.e.
Here the last sum is nonnegative because \(r_e\) and the scalar
positive-part map preserve the sign of differences. As \(v(0)=0\), the
trajectory remains in \(K\).

For a split step, we define \(H_h(x)=x+hF(x)\). On \(K\), each edge
stiffness satisfies \(a_e+3c_e(Bx)_e^2\le4s\), and the chain has degree at
most two. Therefore, if \(h\le1/(1+8s)\),
the Jacobian of \(H_h\) is componentwise nonnegative. Since
\(H_h(\boldsymbol1)=\boldsymbol1+h(b-\boldsymbol1)\le\boldsymbol1\), we have \(\max\{H_h(x),0\}\in K\) for every \(x\in K\).
The inexact correction in Algorithm~\ref{alg:spring-split} is a convex
combination of this exact shadow and \(q=\boldsymbol1/2\), so it also
belongs to \(K\). All tested split rows satisfy this sufficient step
condition, and no upper clipping is performed.

\paragraph{Experiment I: robustness to projection errors}
With \(c=0.1\), \(h=T/N\) and even \(N\), we compare
\[
\varepsilon_k^{\rm ex}=0,\qquad
\varepsilon_k^{\rm dist}=c^2h^3,\qquad
\varepsilon_k^{\rm sp}=c^2Th^2\mathbf1_{\{k=N/2\}}.
\]
Both inexact allocations satisfy \(E_\nu(T)=c^2Th\to0\), whereas \(Q_\nu^{\rm dist}(T)=c^2h\to0\) and \(Q_\nu^{\rm sp}(T)=c^2T=0.03\).
We use \(N=1024,2048,4096,8192\) for \(s=1,10\) and
\(N=4096,8192,16384,32768\) for \(s=100\).

The common localization set \(K\) and Theorem~\ref{thm:limit_solution}
give subsequential convergence under \(h_\nu\to0\) and \(E_\nu(T)\to0\), and
uniqueness identifies every limit, which gives the convergence of the
whole family, including for the sparse allocation. On \(K\), we may use \(\widehat M_K:=\|b\|+\sqrt{12}+4s\sqrt{11}\ge M_K\).
With \(\ell_K=-1\), the constants of Theorem~\ref{thm:trajectory-rate}
satisfy \(\alpha=1\), \(\beta=9M_K^2\), and \(\chi=1+3T\). Its
accumulated-error estimate therefore yields
\[
\|x_\mu-x\|_{L^\infty(0,T)}
\le e^{T/2}
\sqrt{9\widehat M_K^2Th+(1+3T)E_\nu(T)}.
\]
Thus all three allocations have the proved \(O(\sqrt h)\) bound.

The uniform errors of the piecewise affine interpolants are estimated by
sampling against a reference trajectory, as described in the supplementary
material.
All 36 runs remain feasible and satisfy the projection inequalities to
floating-point precision. The sampled errors decrease under refinement
(Figure~\ref{fig:noncompact-results}, left). The finest-pair observed
orders are approximately one for exact and sparse corrections, and
\(0.67\), \(0.68\), and \(0.61\) for distributed corrections at
\(s=1,10,100\), respectively. The sampled maximum errors of the sparse and
exact corrections coincide, because their maxima occur before the sparse
correction, and the later nonzero perturbation does not produce a larger
sampled maximum. This behavior is consistent with the theoretical upper
bound.

To make the sparse perturbation visible, we also plot
\(\Delta_h(t):=\|x_\mu^{\rm sp}(t)-x_\mu^{\rm ex}(t)\|\) and record the
realized residual energy \(\mathcal R_\nu(T):=\sum_{k<N_\nu}h_{\nu,k}\|\rho_{\nu,k}\|^2\le E_\nu(T)\).
Figure~\ref{fig:projection-diagnostics} displays both diagnostics. For
\(N=4096\), \(\Delta_h\) vanishes on \([0,T/2]\) and has maximum values
\(3.95\times10^{-6}\), \(4.13\times10^{-6}\) and \(5.69\times10^{-6}\) for
\(s=1,10,100\), respectively. These maxima equal the maximum nodal
differences, since the difference of the interpolants is affine on each
common grid interval and the norm is convex. Across the tested rows,
\(\mathcal R_\nu(T)/E_\nu(T)\) is approximately \(13.7\)--\(26.9\%\) for
distributed corrections and \(0.0968\)--\(0.2013\%\) for sparse corrections.
Thus the sparse allocation produces a nonzero perturbation while using only
a small fraction of its residual-energy budget. It illustrates convergence
with vanishing \(E_\nu(T)\) and nonvanishing \(Q_\nu(T)\); it does not test
the sharpness of the bound. Equal tolerance budgets need not produce equal
residual energies or equal realized trajectory errors.

\begin{figure}[H]
\centering
\includegraphics[width=\linewidth]{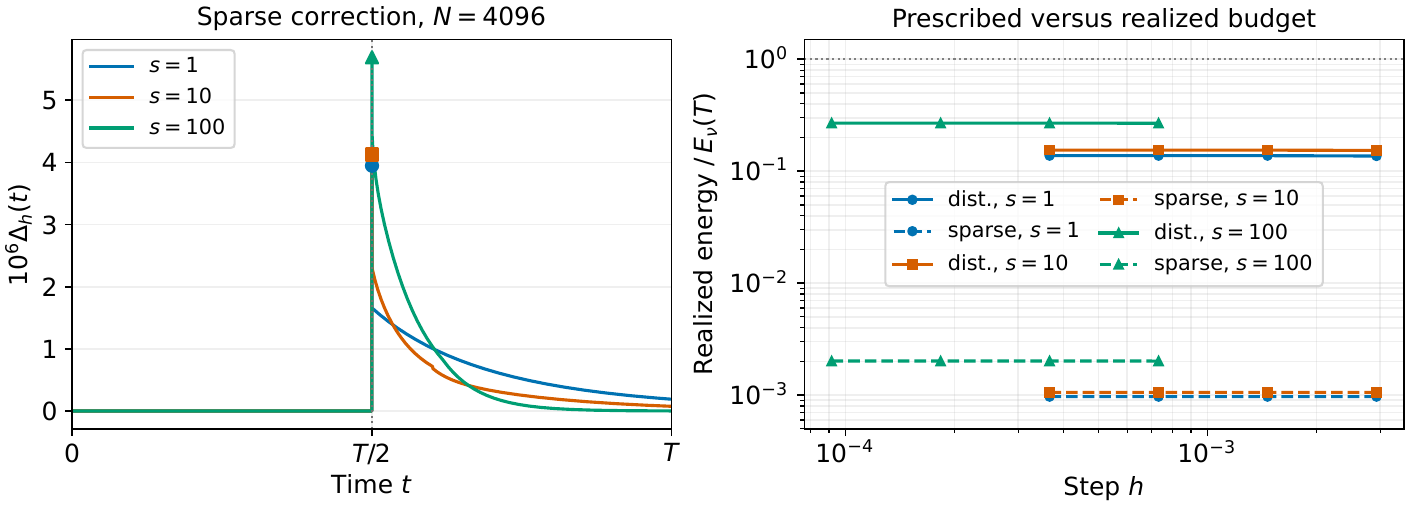}
\caption{Experiment I. Left: the difference \(\Delta_h\) for \(N=4096\) and
\(s=1,10,100\); the sparse correction acts at \(T/2\). Right: the ratio
\(\mathcal R_\nu(T)/E_\nu(T)\); the dotted line marks the ratio one.}
\label{fig:projection-diagnostics}
\end{figure}

\paragraph{Experiment II: computational effect of splitting}
We compare Algorithm~\ref{alg:spring-split}, with exact projections
\(\varepsilon_k=0\), and Algorithm~\ref{alg:spring-resolvent}. The
diagnostic checks found no solver failure or exhausted backtracking search
in the selected runs. Split meshes start within \(h\le(1+8s)^{-1}\), whereas
resolvent meshes start at \(N=32\) without this restriction, so the
comparison concerns the certified split range rather than the largest
empirically stable split steps. The calibration of tolerances and meshes,
and the timing protocol, which reports medians over five randomly
interleaved runs, are described in the supplementary material.

\begin{table}[H]
\centering
\small
\caption{Median runtimes at matched error targets. The ratio is the
resolvent time divided by the split time; values below one favor the
resolvent. Here \(\tau\) is the resolvent solver tolerance.}
\label{tab:noncompact-comparison}
\begin{tabular}{rrrrrrrr}
\toprule
$s$ & Target & $N_{\rm split}$ & $N_{\rm res}$ & $\tau$
& Split ms & Res.\ ms & Ratio\\
\midrule
1 & $10^{-2}$ & 256 & 512 & $10^{-4}$ & 1.84 & 37.55 & 20.46\\
1 & $10^{-3}$ & 4096 & 4096 & $10^{-9}$ & 29.39 & 298.38 & 10.15\\
10 & $10^{-2}$ & 512 & 512 & $10^{-4}$ & 3.61 & 36.51 & 10.12\\
10 & $10^{-3}$ & 4096 & 4096 & $10^{-7}$ & 29.25 & 292.76 & 10.01\\
100 & $10^{-2}$ & 4096 & 512 & $10^{-4}$ & 30.14 & 25.35 & 0.84\\
100 & $10^{-3}$ & 4096 & 4096 & $10^{-7}$ & 30.14 & 283.86 & 9.42\\
\bottomrule
\end{tabular}
\end{table}

\begin{figure}[H]
\centering
\includegraphics[width=\linewidth]{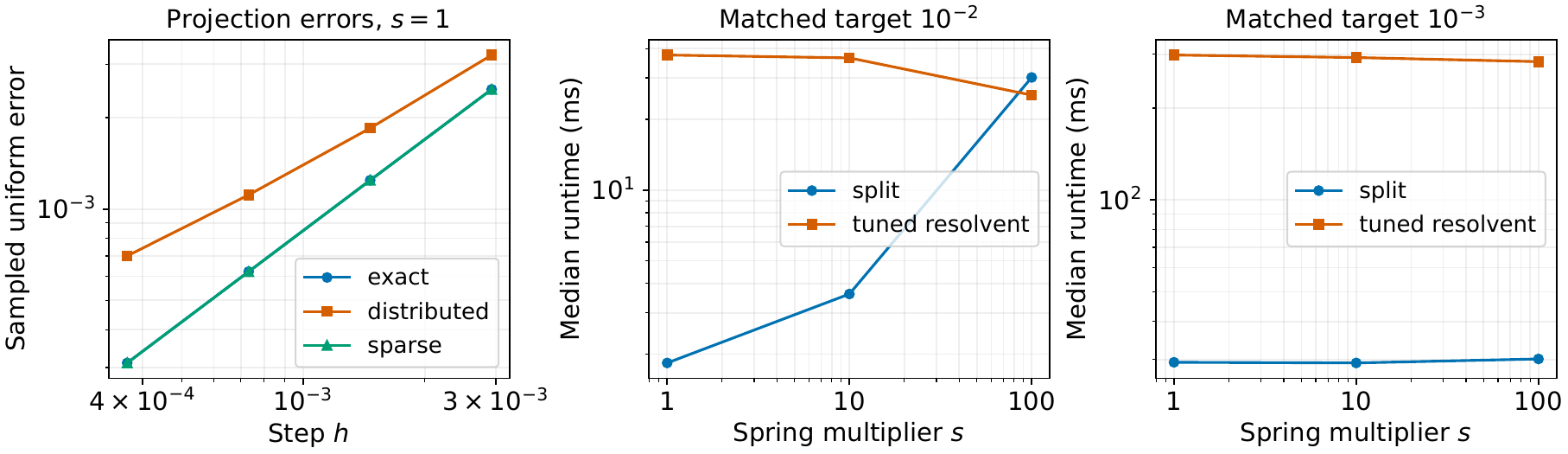}
\caption{Left: Experiment I, sampled errors for \(s=1\); the exact and
sparse curves coincide. Middle and right: Experiment II, median runtimes at
the error targets \(10^{-2}\) and \(10^{-3}\).}
\label{fig:noncompact-results}
\end{figure}

Figure~\ref{fig:noncompact-results} (middle and right) and
Table~\ref{tab:noncompact-comparison} show the median runtimes. At the
target \(10^{-3}\), both methods use \(4096\) steps, and splitting is
approximately 9--10 times faster in these implementations. At \(s=100\) and
the target \(10^{-2}\), the resolvent uses \(512\) steps against \(4096\) in
the certified split search, and it is approximately \(1.19\) times faster.
The corresponding sampled errors are \(6.71\times10^{-3}\) and
\(6.85\times10^{-4}\), respectively, so the comparison matches prescribed
targets rather than identical attained errors. In this case, the
calibration of the tolerances changes the ranking, which confirms that the
computational performance depends on the stiffness, the requested accuracy
and the implementation.

\paragraph{Numerical checks and interpretation}
The reference, sampling, and solver-tolerance checks are given in the
supplementary material; in particular, at \(s=100\) and the target
\(10^{-2}\), the resolvent runs still meet the target when the solver
tolerance is tightened. The companion data include the full sweeps, the
timing samples, and the implementation metadata.

The application exhibits confinement on a noncompact constraint set
without global linear growth, and the robustness experiment illustrates the
control of accumulated errors. The efficiency experiment shows that the
split method is faster in five of the six tested cases, while the full
resolvent is faster for the stiffest network, \(s=100\), at the target
\(10^{-2}\). Neither experiment establishes a better general convergence
order or a universal superiority over the full resolvent. Moreover, exact
full-resolvent steps dissipate the loaded constrained energy
\(\Psi-\langle b,\cdot\rangle\) for every step size.

\section{Conclusions and future research}
\label{sec:conc}

We studied a catching-up scheme with approximate projections for
differential inclusions governed by a maximally monotone operator and a
continuous perturbation. Through the canonical decomposition
\(A=G+N_C\), the problem becomes a sweeping process driven by the
effective field \(F=f-G\). On a noncompact constraint set, tangent
dissipativity at a feasible center confines the discrete trajectories to a
fixed ball for small steps and small projection errors. These bounds are
uniform in time and allow fields with superlinear growth\rev{, which a globally
Lipschitz bound on the minimal norm of the perturbation excludes}. Within the
resulting bounded regions, the accumulated budget \(E_\nu(T)\) controls the
projection residuals and accommodates sparse inaccuracies. On this basis,
we proved the convergence of the scheme \rev{on each finite horizon, up to a
subsequence. For locally Lipschitz perturbations, we also proved global
well-posedness and\rev{, for a common initial state,} a} trajectory-error bound of order
\(\sqrt{h_\nu+E_\nu(T)}\) on each finite horizon.

The elastic network of Section~\ref{sec:noncompact-springs} illustrates this
error control. Its runtime comparison of the split and full-resolvent
discretizations at the same error target is specific to the implementations
and parameters tested. Our future research will address time-dependent
operators and moving constraint sets.


\section*{Disclosure statement}

The authors report there are no competing interests to declare.

\section*{Funding}

\rev{This research received no external funding.}

\section*{Data availability statement}

The code and the numerical data supporting the experiments are provided as
supplementary material, including the projection-error diagnostics, the
resolvent tolerance study, the timing samples, and instructions for
reproducing the reported results.

\bibliographystyle{tfq}
\bibliography{References}

\end{document}